\documentclass[12pt,a4paper,twoside,reqno]{amsart}

\usepackage{amsmath,amssymb,amsfonts,amsthm,mathrsfs}
\usepackage{enumitem}
\usepackage{times,hyperref,color}
\usepackage{graphicx}
\usepackage{cancel}

\usepackage{cite}
\usepackage[toc,page]{appendix}
\usepackage{bm}
\usepackage{tikz}
\usetikzlibrary{cd}
\usepackage{accents}
\usepackage{float}

\newfloat{diagram}{htbp}{lod}
\floatname{diagram}{Diagram}

\newcommand{\G}{\mathcal{G}}
\newcommand{\bz}{\bm{z}}

\newcommand{\n}{{}^{(n)}}

\newcommand{\B}{\mathcal{B}}
\newcommand{\bZ}{\bm{Z}}
\newcommand{\jet}{\mathrm{j}}
\newcommand{\hB}{\widehat{\B}}
\newcommand{\hrho}{\widehat{\rho}}
\newcommand{\ii}{^{(\infty)}}

\newcommand{\ba}{{\bm{a}}}
\newcommand{\bb}{{\bm{b}}}

\newcommand{\Dt}{{\rm D}}
\newcommand{\dt}{{\rm d}}

\newcommand{\re}{{\rm Re}\,}
\newcommand{\im}{{\rm Im}\,}

\newcommand{\omu}{\overline{\mu}}
\newcommand{\bmu}{\bm{\mu}}

\newcommand{\ow}{\overline{w}}
\newcommand{\oW}{\overline{W}}
\newcommand{\oxi}{\overline{\xi}}
\newcommand{\pp}[2]{\frac{\partial #1}{\partial #2}}
\newcommand{\vv}{\mathbf{v}}
\let\mathcal\mathscr

\newtheorem{Theorem}{Theorem}[section]

\newtheorem{Lemma}[Theorem]{Lemma}

\theoremstyle{definition}
\newtheorem{Definition}[Theorem]{Definition}

\newtheorem{Remark}[Theorem]{Remark}

\newcommand{\oz}{\bar{z}}
\newcommand{\oZ}{\overline{Z}}

\def\hexnumber#1{\ifcase#1 0\or1\or2\or3\or4\or5\or6\or7\or8\or9\or
 A\or B\or C\or D\or E\or F\fi}

\edef\msbhx{\hexnumber\symAMSb}   

\mathchardef\emptyset="0\msbhx3F
\def\i{\,{\rm i}\,}

\makeatletter
\@namedef{subjclassname@2020}{%
  \textup{2020} Mathematics Subject Classification}
\makeatother


\allowdisplaybreaks[4]

\begin{document}

\title{
Holomorphic normal forms of six-dimensional
\\ totally nondegenerate CR manifolds in $\bm{\mathbb{C}^5}$
}

\author{Masoud Sabzevari}
\address{Department of Pure Mathematics, Faculty of Mathematics and Statistics, University of Isfahan, 81746-73441 Isfahan, Iran and School of
Mathematics, Institute for Research in Fundamental Sciences (IPM), 19395-5746, Tehran, Iran}
\email{sabzevari@ipm.ir}

\date{\number\year-\number\month-\number\day}

\begin{abstract}
The class of six-dimensional totally nondegenerate CR submanifolds in $\mathbb C^5$ contains an {\it infinite number} of models parameterized by nonzero pairs $(\ba, \bb)\in\mathbb C\times\mathbb R$ appearing in their defining equations. Employing the equivariant moving frame method, we construct normal forms of this class.  The applied normalizations yield into fourteen biholomorphically inequivalent subclasses, each with its own normal form. We also determine the Lie algebra of infinitesimal CR automorphisms associated with each subclass. Finally, by applying the method of involution, we prove the convergence of the constructed normal forms.
\end{abstract}

\maketitle

\pagestyle{headings} \markright{Holomorphic normal form of totally nondegenerate manifolds in $\mathbb C^5$}
\numberwithin{equation}{section}

\section{Introduction}

The theory of Cauchy--Riemann (CR for short) manifolds has a rich history dating back to the 1907 work of Henri Poincar\'e, \cite{Poincare-1907}, who showed that two real hypersurfaces in $\mathbb{C}^2$ are not in general biholomorphically equivalent.  The problem of finding the invariants to distinguish these real hypersurfaces was then solved by Cartan in \cite{Cartan-1932}, using his equivalence method of coframes.  About four decades later, Chern and Moser extended these results to arbitrary dimensions by developing the theory of {\it normal forms} for Levi nondegenerate hypersurfaces. For a modern treatment of normal forms and a survey on related open problems in CR geometry, we refer the reader to \cite{KKZ-17}.  Around the same time, Tanaka independently gave a different extension to higher dimensions using, what is now referred as, Tanaka's theory, \cite{Tanaka-1962,Tanaka-1976}.

For an arbitrary real manifold $M$ of dimension $\geq 2$, let $T^cM$ be an even dimensional sub-distribution of its tangent bundle $TM$ equipped with a fiber preserving complex structure $J\colon T^cM\rightarrow T^cM$, satisfying $J\circ J=-{\rm id}$. By definition, see \cite{BER}, the manifold $M$ is called an ({\it abstract}) {\it CR manifold}, with CR structure $T^cM$, if the following two conditions are satisfied:
\begin{itemize}
\item[a)] the intersection $T^{1,0}M\cap T^{0,1}M$ is trivial, where $T^{1,0}M = \{X-\i J(X): X\in T^cM\}$ and $T^{0,1}M=\overline{T^{1,0}M}$;
\item[b)] the holomorphic distribution $T^{1,0}M$ enjoys the Frobenius condition  $[T^{1,0}M, T^{1,0}M]\subset T^{1,0}M$.
\end{itemize}
For a CR manifold $M$, the integers $\frac{1}{2} {\rm rank}(T^cM)$ and ${\rm dim}_{\mathbb R}M-{\rm rank}(T^cM)$ are called the CR dimension and codimension of $M$.

Let $\mathfrak g^{-1} = T^cM$ and define $\mathfrak g^{-t} = [\mathfrak g^{-t+1}, \mathfrak g^{-1}]$ for every $t\geq 1$. In this successive construction, assuming that $\ell$ is the first integer with $\mathfrak g^{-\ell}=\{0\}$, the CR manifold $M$ is said to be {\it regular of depth} $\ell$ if the so-called {\it Tanaka symbol} (cf. \cite{Sab-MZ}):
\begin{equation}\label{Tanaka-Sym}
\mathfrak g^{-1}+\mathfrak g^{-2}+\ldots+\mathfrak g^{-\ell},
\end{equation}
equipped with the standard Lie bracket, forms a locally graded Lie algebra which spans the entire tangent space $TM$.  By definition, a regular CR manifold is {\it totally nondegenerate} whenever the truncation of its associated Tanaka symbol \eqref{Tanaka-Sym} with its last nonzero component $\mathfrak g^{-\ell}$ is isomorphic to the depth $\ell-1$ free algebra generated by $\mathfrak g^{-1}$.  We refer the reader to \cite{Beloshapka2004, Sab-MZ} for more details.

Totally nondegenerate CR manifolds were first introduced by Valerii Beloshapka in \cite{Beloshapka2004}, as part of his investigation on Poincar\'{e}-Chern-Moser {\it model surfaces}, \cite{Poincare-1907, Chern-Moser} (see also \cite{Kruglikov-26}). Beloshapka in \cite{Beloshapka2004}, associated  model surfaces to each class of totally nondegenerate CR manifolds of arbitrary CR dimension and codimension. 

In the fixed CR dimension one, the models associated to each class of totally nondegenerate CR manifolds of  codimensions one, two, and three are {\it unique}. In codimension four, however, the situation is remarkably different. Indeed, this class, which consists of six-dimensional totally nondegenerate CR submanifolds of $\mathbb C^5$, admits an {\it infinite number} of model surfaces.  As is shown in \cite{Beloshapka2006, Sab-SCM}, after applying preliminary normalizations, such CR submanifolds may be represented in the local coordinates $z, w^r = u_r+\i v^r, r = 1,\ldots, 4$, of $\mathbb C^5$ as the graph of four analytic functions
\begin{equation}\label{def-eq}
\begin{cases}
v^1=z\oz+{\rm O}_3(z,\oz,u), \\
v^2=\frac{1}{2}\,(z^2\oz+z\oz^2)+{\rm O}_4(z,\oz,u),\\
v^3=-\frac{\i}{2} \, (z^2\oz-z\oz^2)+{\rm O}_4(z,\oz,u),\\
v^4=\ba\, z^3\oz+\overline\ba\, z\oz^3+\bb\, z^2\oz^2+{\rm O}_5(z,\oz,u),
\end{cases}
\end{equation}
where $\ba\in\mathbb C$, $\bb\in\mathbb R$, and $(\ba, \bb)\neq (0,0)$.  Assigning the weights $[z] = [\oz] = 1, [u_1] = 2, [u_2] = [u_3] = 3$, and $[u_4]=4$ to the independent variables, the terms ${\rm O}_t(z,\oz,u)$ in \eqref{def-eq} denote sums taken over all monomials of weights $\geq t$. The lowest order nonzero monomials appearing in the defining equations \eqref{def-eq} ensure the total nondegeneracy of the manifold. Associated to the class of totally nondegenerate CR manifolds \eqref{def-eq}, we have Beloshapka's {\it model surfaces}
\begin{equation}\label{def-eq-model}
M(\ba, \bb):
\begin{cases}
v^1=z\oz,\\
v^2=\frac{1}{2}\,(z^2\oz+z\oz^2),\\
v^3=-\frac{\i}{2}\, (z^2\oz-z\oz^2),\\
v^4=\ba\, z^3\oz+\overline\ba\, z\oz^3+\bb\, z^2\oz^2,
\end{cases}
\end{equation}
parameterized by the pairs $(\ba, \bb)$.
These model surfaces --- which admit maximal possible CR symmetry dimensions among totally nondegenerate CR manifolds \eqref{def-eq} --- are not necessarily biholomorphically inequivalent. As is shown in \cite{Beloshapka2006, Mamai-13, Sab-SCM}, the equivalence between the models $M(\ba, \bb)$ with $\bb\neq 0$ --- and thus their associated {\it moduli spaces} --- is identified by the single real invariant
\begin{equation}\label{J}
\mathfrak{I}=\frac{\ba\overline{\ba}}{\bb^2}.
\end{equation}

The construction of normal forms and the solution to the equivalence problem for  three-dimensional (totally) nondegenerate CR manifolds in $\mathbb C^2$ are accomplished in the seminal works \cite{Cartan-1932,Chern-Moser}.  In $\mathbb C^3$, normal forms for the class of four-dimensional totally nondegenerate submanifolds was constructed by Beloshapka, Ezhov and Schmalz in \cite{BES}. Furthermore, the equivalence problem for five-dimensional totally nondegenerate submanifolds in $\mathbb C^4$ was solved by Merker and the author in \cite{5-cubic}, while their normal forms were obtained in \cite{Sab-JGA}.

The aim of the present paper is to construct normal forms for the aforementioned $6$-dimensional totally nondegenerate CR submanifolds in $\mathbb C^5$. The emergence of an infinite family of model surfaces within this class renders its normal form classification more intriguing and, at the same time, more challenging. As we will see, the desired normal forms fall into fourteen biholomorphically inequivalent subbranches displayed in Diagram \ref{branches}.

\begin{diagram}[htbp]
\begin{center}
\begin{tikzpicture}
\draw[thick] node[left] {{}} (0,0) -- (3/2,2.2) node[right] {{\scriptsize Branch A$^\prime$}};
\draw[thick] (0,0) -- (3/2,0) node[right] {{\scriptsize Branch A$^{\prime\prime}$}};
\draw[thick] (0,0) -- (3/2,-2.2) node[right] {{\scriptsize Branch B}};
\draw[thick] (3,2.2) -- (4,2.9) node[right] {{\scriptsize Branch A$^\prime$-1}};
\draw[thick] (3,2.2) -- (4,1.5) node[right] {{\scriptsize Branch A$^\prime$-2}};
\draw[thick] (3.1,0) -- (4,.7) node[right] {{\scriptsize Branch A$^{\prime\prime}$-1}};
\draw[thick] (3.1,0) -- (4,0) node[right] {{\scriptsize Branch A$^{\prime\prime}$-2}};
\draw[thick] (3.1,0) -- (4,-.7) node[right] {{\scriptsize Branch A$^{\prime\prime}$-3}};
\draw[thick] (5.9,-.72) -- (6.25,-.3) node[right] {{\scriptsize Branch A$^{\prime\prime}$-3-1}};
\draw[thick] (5.9,-.72) -- (6.25,-1.1) node[right] {{\scriptsize Branch A$^{\prime\prime}$-3-2}};
\draw[thick] (8.4,-1.1) -- (9,-.6) node[right] {{\scriptsize Branch A$^{\prime\prime}$-3-2-1}};
\draw[thick] (8.4,-1.1) -- (9,-1.6) node[right] {{\scriptsize Branch A$^{\prime\prime}$-3-2-2}};
\draw[thick] (2.9,-2.2) -- (4,-1.5) node[right] {{\scriptsize Branch B-1}};
\draw[thick] (2.9,-2.2) -- (4,-2.9) node[right] {{\scriptsize Branch B-2}};
\draw[thick] (5.8,1.5) -- (6.25,2) node[right] {{\scriptsize Branch A$^\prime$-2-1}};
\draw[thick] (5.8,1.5) -- (6.25,1) node[right] {{\scriptsize Branch A$^\prime$-2-2}};
\draw[thick] (5.7,-2.9) -- (6.25,-2.4) node[right] {{\scriptsize Branch B-2-1}};
\draw[thick] (5.7,-2.9) -- (6.25,-2.9) node[right] {{\scriptsize Branch B-2-2}};
\draw[thick] (5.7,-2.9) -- (6.25,-3.4) node[right] {{\scriptsize Branch B-2-3}};
\draw[thick] (8.3,1) -- (9,1.65) node[right] {{\scriptsize Branch A$^\prime$-2-2-1}};
\draw[thick] (8.3,1) -- (9,0.4) node[right] {{\scriptsize Branch A$^\prime$-2-2-2}};
\draw[thick] (8.25,-3.4) -- (9,-2.9) node[right] {{\scriptsize Branch B-2-3-1}};
\draw[thick] (8.25,-3.4) -- (9,-4.1) node[right] {{\scriptsize Branch B-2-3-2}};
\end{tikzpicture}
\end{center}
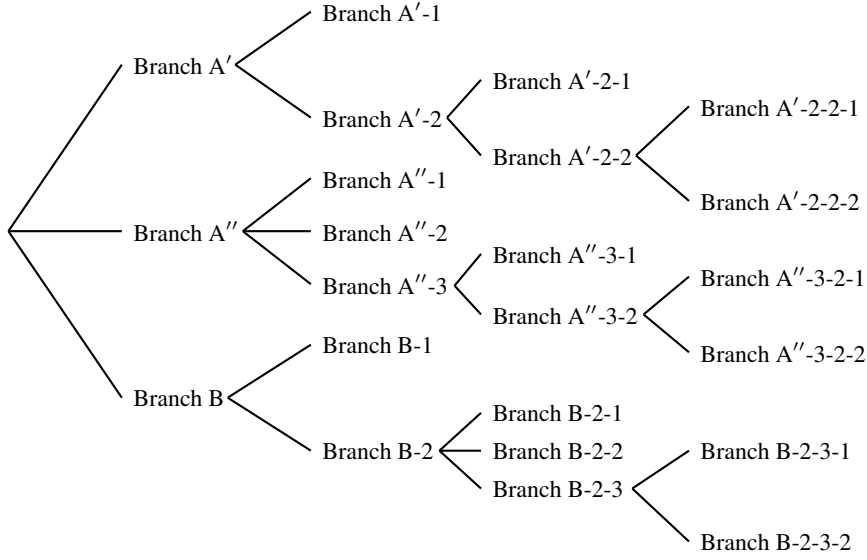
\caption{The $14$ inequivalent branches. }\label{branches}
\end{diagram}

For each branch we construct the associated normal form and determine its holomorphic isotropy group, which is either trivial or of dimension one or two.  The branches {\bf A$^\prime$}, {\bf A}$^{\prime\prime}$ and {\bf B} in Diagram \ref{branches} emerge according to whether $\ba$ and $\bb$ in \eqref{def-eq} vanish or not. Branch {\bf A$^\prime$} corresponds to the case where $\ba$ and $\bb$ are identically nonzero while {\bf A$^{\prime\prime}$} corresponds to $\ba \neq 0$ and $\bb = 0$. Branch {\bf B}, on the other hand, considers the case where $\ba = 0$ and $\bb \neq 0$.

By definition \cite{Ebenfelt-98, Sab-26}, a normal form $N$ of a given manifold $M$ is {\it complete} if the corresponding {\it normal form transformation} $M\rightarrow N$ is unique modulo a {\it finite} dimensional choice of normalizations. In this paper, we find out that

\begin{Theorem}
\label{main-result}
Every $6$-dimensional totally nondegenerate CR submanifold of $\mathbb C^5$ transforms to a {\sl complete} and {\sl convergent} normal form with the defining Taylor series expansions
\begin{equation}\label{NF-complete}
\aligned
v^1&=z\oz+\sum_{j+k+|\ell|\geq 5}\,\frac{V^1_{Z^j\oZ^k U^\ell}}{j! k! \ell !}\,z^j \oz^k u^\ell, \\
v^2&=\frac{1}{2}\,(z^2\oz+z\oz^2)+\sum_{j+k+|\ell|\geq 5}\,\frac{V^1_{Z^j\oZ^k U^\ell}}{j! k! \ell !}\,z^j \oz^k u^\ell,\\
v^3&=-\frac{\i}{2} \, (z^2\oz-z\oz^2)+\frac{V^3_{Z^3\oZ}}{6}\,z^3 \oz+\frac{V^3_{Z\oZ^3}}{6}\,z \oz^3+\sum_{j+k+|\ell|\geq 5}\,\frac{V^1_{Z^j\oZ^k U^\ell}}{j! k! \ell !}\,z^j \oz^k u^\ell,\\
v^4&=\ba\, z^3\oz+\overline\ba\, z\oz^3+\bb\, z^2\oz^2+\frac{V^4_{Z^2\oZ U_1}}{2}\,z^2 \oz u_1+\frac{V^4_{Z\oZ^2 U_1}}{2}\,z \oz^2 u_1+\frac{V^4_{Z^2\oZ U_2}}{2}\,(z^2 \oz u_2- z \oz^2 u_2)\\
& \ \ \ \ \ \ \ \ \ \ \ \ \ \ +\sum_{j+k+|\ell|\geq 5}\,\frac{V^4_{Z^j\oZ^k U^\ell}}{j! k! \ell !}\,z^j \oz^k u^\ell,
\endaligned
\end{equation}
where its Taylor coefficients are subject to the following {\sl general normal form constraints} --- here, $\delta$ denotes the Kronecker symbol
\[
\aligned
V^r_{Z\oZ U^\ell} = V^r_{Z^{j+1}U^\ell} = V^r_{U^\ell} = V^2_{Z^2\oZ U^\ell} = V^3_{Z^2\oZ U_2^j U_3^k U_4^l} =  V^1_{Z^{j+2}\oZ U^\ell}
= V^1_{Z^2\oZ{}^2 U_2^j U_4^l}=V^1_{Z^2\oZ^2 U_2 U_3^k U_4^l}=
\\
 V^2_{Z^3\oZ U^\ell}= V^2_{Z^2\oZ{}^2 U_2^j U_3^k U_4^l}
=V^3_{Z^2\oZ U_1U_2^j U_3^k U_4^l}
= V^3_{Z^2\oZ{}^2U_3^k U_4^l}
=V^4_{Z^2\oZ U_3^{k} U_4^l} = \re V^4_{Z^2\oZ U_2^{j+1} U_3^k U_4^l}
 = 0,
\endaligned
\]
for $j,k,l\in\mathbb N_0:=\mathbb N\cup\{0\}$, $\ell\in\mathbb N_0^4$ and $r=1,\ldots, 4$. Additionally, in both branches {\bf A$^\prime$} and {\bf A$^{\prime\prime}$}, the normal form enjoy the supplementary constraints
\[
\ba=\frac{\i}{6}, \qquad \im V^2_{Z^2\oZ^3 U_4^l}=V^3_{Z^3\oZ U_4^l}=V^3_{Z^3\oZ U_1 U_4^l}=\re V^3_{Z^3\oZ U_2 U_4^l}=V^4_{Z^3\oZ U_3^k U_4^l}=0,
\]
together with $\im V^4_{Z^2\oZ U_2 U_4^l}=0$ in Branch {\bf A$^\prime$} and $\im V^2_{Z^4\oZ U_4^l}=0$ in Branch {\bf A$^{\prime\prime}$}, respectively. For Branch {\bf B}, the supplementary normal form constraints read
\[
\aligned
\bb=\frac{1}{4}, \qquad & V^2_{Z^2\oZ^2 U_1 U_3^{k}}=V^3_{Z^2\oZ^2 U_1 U_4^l}=V^4_{Z^2\oZ^2 U_3^{k}}=V^4_{Z^2\oZ U_1U_3^{k} U_4^l}=V^4_{Z^2\oZ^2 U_2 U_4^l}=V^4_{Z^2\oZ^2 U_1 U_4^l}=
\\
&V^4_{Z^2\oZ^2 U_4^l}=\im V^4_{Z^2\oZ U_2 U_4^l}=0.
\endaligned
\]
\end{Theorem}

The main result, stated above, is obtained upon completion of computations in the major branches {\bf A$^\prime$}--{\bf A$^{\prime\prime}$} and {\bf B}. The normal form transformations leading to the above complete normal forms are unique upon a $1$-dimensional family of normalization choices in Branches {\bf A$^\prime$} and {\bf A$^{\prime\prime}$}, whereas in Branch {\bf B}, this freedom is of dimension $2$. Although completeness of the normal forms is already attained in these major branches, our next objective in the subsequent subbranches of Diagram \ref{branches} is to determine whether the finite number of remaining free parameters of the holomorphic group action can be further normalized. In fact, the ultimate goal is to achieve the {\it maximal possible normalization} of these group parameters.

In addition to the above result, we also realize in Branches {\bf A$^\prime$} and {\bf A$^{\prime\prime}$} that the dimension of the holomorphic isotropy group is $\leq 1$ with the maximal dimension, that is one, achieved by the normalized Beloshapka's models $M(\frac{\i}{6}, \bb)$ in \eqref{def-eq-model} with $\bb\geq 0$. This isotropy group is generated by the {\it dilations}
\begin{equation}\label{dilation}
Z = e^tz,\quad W_1=e^{2t}w_1,\quad W_2=e^{3t}w_2,\quad W_3=e^{3t}w_3,\quad W_4=e^{4t}w_4,\qquad
t\in\mathbb R.
\end{equation}

In Branch {\bf B}, the dimension of the isotropy groups is $\leq 2$ with, again, the maximal dimension only attained by the holomorphically unique model surface  $M(0,\frac{1}{4})$ of this branch. The isotropy group of this model is generate by the dilation transformations \eqref{dilation} along with the {\it rotations}
 \begin{equation}\label{rotation}
\begin{gathered}
Z=e^{-{\rm i}\theta} z,\qquad W_1=w_1,\qquad W_4=w_4,\\
W_2 = \cos(\theta)\, w_2- \sin(\theta)\, w_3,\qquad
W_3 = \cos(\theta)\, w_3+ \sin(\theta)\, w_2,
\end{gathered}
\qquad \theta\in\mathbb R.
\end{equation}
In all cases, Beloshapka's models are the only surfaces with maximal CR symmetry algebra which confirms --- and slightly improves --- Proposition 12 of \cite{Beloshapka2004} in CR dimension one and codimension four.

In this paper, our principal tool employed for the construction of the desired normal forms is the modern and powerful theory of equivariant moving frames, developed by Peter Olver and his school \cite{Olver-Fels-99, Olver-Pohjanpelto-08, Olver-2005, Olver-2018}. This far-reaching reformulation of Cartan's classical moving frames, furnishes a {\it systematic} procedure which simultaneously solves equivalence problems --- represented in Cartan's sense by fundamental invariants and structure equations \cite{Cartan-1935} --- and constructs their underlying normal forms in the sense of Moser \cite{Chern-Moser}. Indeed, it establishes a concrete and illustrative bridge between Cartan's theory of equivalences and Moser's theory of normal forms. In contrast to the classical methods (see e.g. \cite{BES, Chern-Moser}), the normal form construction proposed by the equivariant moving frame theory is purely {\it symbolic}, involving only linear algebra without requiring to explicit computation of the emerging invariants. We refer the reader to Section \ref{Sec-prel} for more details and explanations.

As stated in Theorem \ref{main-result}, our normal  forms \eqref{NF-complete} are also convergent. Historically, establishing the convergence of geometric normal forms under the action of pseudo-groups is a substantial challenge which requires foundational and technically involved arguments. The most well-known approach to this problem originates in the classical and celebrated work of Chern and Moser \cite{Chern-Moser}, who  introduced a distinguished class of curves, referred to as {\it chains}, for analyzing the convergence issue (see also \cite{Kossovskiy-Zaitsev-19, Makhmali-26}). In contrast, we have obtained more recently in \cite{OSV-preprint} a new criteria for the convergence of normal forms, based on the modern algebraic-differential theory of {\it involution} and {\it Pommaret bases} \cite{Seiler}. In the present paper, we apply the results from \cite{OSV-preprint} to prove the desired convergence of our normal forms.

The outline of this paper is as follows. In Section \ref{Sec-prel} we provide the necessary background to construct normal forms using the equivariant moving frame method.  By exploiting the recurrence relations, our computations are performed symbolically without relying on the coordinate expressions of the differential invariants or the moving frames. In Section \ref{sec-normal-form}, we begin the normal form construction. Preliminary normalizations are done up to order four, at which point the problem splits into the three branches {\bf A$^\prime$}, {\bf A}$^{\prime\prime}$ and {\bf B}. Moreover, we will see in this section how the total nondegeneracy assumption is recast within the equivariant moving frame setting. Throughout the two sections \ref{sec-Branch-A} and \ref{sec-Branch-B}, we construct the normal forms for each subbranches appearing in Diagram \ref{branches}. Section \ref{sec-convergence} is devoted to the proof of the convergence of the constructed normal forms by applying our recent involution criteria \cite{OSV-preprint}. As required by the main results of this paper, we will restrict our normal forms to be {\it minimal} which, intuitively, means that each normalization is carried out at the lowest possible order.

\section{Preliminaries}
\label{Sec-prel}

This section provides the necessary tools for computing normal forms via the equivariant moving frame method.  We refer the reader to the foundational papers \cite{Olver-Fels-99,Olver-Pohjanpelto-05,Olver-Pohjanpelto-08} for more details.  First, we introduce the underlying pseudo-group action.

\subsection{The holomorphic pseudo-group}

Let $\mathcal G$ denotes the pseudo-group of local holomorphic automorphisms of $\mathbb C^5$ with local coordinates $z, w^1, w^2, w^3, w^4$. By definition, an invertible map $(z,w^1, w^2, w^3, w^4)\mapsto (Z, W^1, W^2, W^3, W^4)$ belongs to $\mathcal G$ provided
\begin{equation*}
Z_{\oz} = Z_{\ow^r} = W^r_{\oz} = W^r_{\ow^s} = 0\qquad \text{with}\qquad r,s = 1,\ldots,4.
\end{equation*}
Expressing $w^r = u_r + \i v^r$ into its real and imaginary parts and letting $u = (u_1,\ldots,u_4)$, $v=(v^1,\ldots,v^4)$, the transformed variables $W^r$ split into their real and imaginary parts
\begin{equation*}
W^r(z,u,v) = U_r(z,\oz,u,v) + \i V^r(z,\oz,u,v),\qquad r = 1,\ldots,4.
\end{equation*}
In the following, we let $\oZ = \overline{Z(z,u,v)}$ and $\oW^j=\overline{W^j(z,u,v)}$ denote the complex conjugate of the functions.  Following the same argument as in \cite[$\S$3]{Sab-JGA}, an invertible map $(z, w)\mapsto (Z, W)$ belongs to the Lie pseudo-group $\mathcal G$ provided it satisfies the determining Cauchy-Riemann equations
\begin{equation}
\label{eq: determining equations}
\begin{gathered}
Z_{\oz} = \oZ_z = 0, \qquad
Z_{v^r} = \i Z_{u_r}, \qquad
\oZ_{v^r} = -\i\oZ_{u_r},
\\
\frac{\partial U_r}{\partial z} = \i V^r_z,\qquad
\frac{\partial U_r}{\partial \oz}=-\i V^r_{\oz}, \qquad
\frac{\partial U_r}{\partial v^s}=-V^r_{u_s},\qquad
V^r_{v^s} = \frac{\partial U_r}{\partial u_s},
\end{gathered}
\end{equation}
for $r, s= 1,\ldots,4$.  At the infinitesimal level, we introduce the vector field
\begin{equation}\label{eq: v}
\vv = \xi(z,u,v)\pp{}{z} + \overline{\xi}(\oz,u,v)\pp{}{\oz} + \sum_{r=1}^4\bigg( \eta^r(z,\oz,u,v)\pp{}{u_r} + \phi^r(z,\oz,u,v)\pp{}{v^r}\bigg),
\end{equation}
where $\overline{\xi}(\oz,u,v)= \overline{\xi(z,u,v)}$, $\eta^r = \frac{1}{2}(\zeta^r+\overline{\zeta}{}^r)$ and $\phi^r = \frac{{\rm i}}{2}(\overline{\zeta}{}^r-\zeta^r)$ for some holomorphic functions $\zeta^1(z,w)$, $\ldots$, $\zeta^{4}(z,w)$. Linearizing the determining equations \eqref{eq: determining equations} at the identity transformation (cf. \cite{Olver-1995, Olver-Pohjanpelto-05}), the vector field \eqref{eq: v} belongs to the (local) Lie algebra $\mathfrak{g}=\mathfrak{hol}(\mathbb C^5)$ if and only if its vector components satisfy, up to order two, the \emph{infinitesimal determining equations}
\begin{equation}\label{eq: infinitesimal determining equations}
\aligned
\xi_{\oz}&=\oxi_z=0,\qquad \xi_{v^r} = \i\xi_{u_r},\qquad \oxi_{v^r}=-\i\oxi_{u_r},
 \\
\phi^r_z &= -\i\eta^r_z, \qquad \phi^r_{\oz} = \i\eta^r_{\oz}, \qquad \phi^r_{u_s} = - \eta^r_{v^s}, \qquad \phi^r_{v^s}=\eta^r_{u_s},
\\
\xi_{\oz,x} &= 0,\qquad \xi_{zv^r} = \i\xi_{zu_r},\qquad \xi_{u_rv^s} = \i\xi_{u_r u_s},\qquad \xi_{v^rv^s} = -\xi_{u_r u_s},
\\
\oxi_{z,x} &= 0, \qquad  \oxi_{\oz v^r} = -\i\oxi_{\oz u_r},\qquad \oxi_{u_rv^s} = -\i\oxi_{u_r u_s},\qquad \oxi_{v^r v^s} = -\oxi_{u_r u_s},
\\
\eta^r_{z\oz} &=0, \qquad \eta^r_{v^s v^t} =-\eta^r_{u_s u_t}, \qquad \eta^r_{zv^s}=\i \eta^r_{zu_s},\qquad
\eta^r_{\oz v^s} = -\i \eta^r_{\oz u_s},
\\
\phi^r_{z,x} &= -\i\eta^r_{z,x}, \qquad \phi^r_{\oz,x} = \i \eta^r_{\oz,x},\qquad \phi^r_{u_s,x} = -\eta^r_{v^s,x},\qquad
\phi^r_{v^s,x} = \eta^r_{u_s,x},
\endaligned
\end{equation}
where $r, s, t=1,2,3,4$, and $x\in\{z,\oz,u,v\}$.  Higher order infinitesimal determining equations are obtained by successive differentiations of the above equations.

Let $\bmu^{(\infty)} = (\mu^z_A, \mu^{\oz}_A, \mu^{u_1}_A,\ldots,\mu^{u_4}_A, \mu^{v^1}_A,\ldots,\mu^{v^4}_A)$ denote the Maurer--Cartan forms of the holomorphic pseudo-group $\G$ where, for $A=(j,k, \ell, \ell')$ with $j,k\in \mathbb{N}_0:=\mathbb{N}\cup \{0\}$, $\ell=(\ell_1, \ldots, \ell_4),\ell'=(\ell'_1, \ldots, \ell'_4) \in \mathbb{N}_0^4$, we denote
\[
\mu^x_A = \mu^x_{Z^j \oZ{}^k U^\ell V^{\ell'}}, \qquad x=z, \oz, u, v,
\]
where $U^\ell:=U_1^{\ell_1}\ldots U_4^{\ell_4}$ and $V^{\ell'}:={V^1}^{\ell'_1}\ldots {V^4}^{\ell'_4}$.
Coordinate expressions for the Maurer--Cartan forms are introduced in \cite{Olver-Pohjanpelto-05}, but they will not be necessary for the construction of normal forms as we view them just {\it symbolically} in this paper. For convenience, we hereafter let
\begin{equation*}
\mu = \mu^z, \qquad \overline\mu = \mu^{\oz}, \qquad \alpha^r = \mu^{u_r}, \qquad \gamma^r = \mu^{v^r}, \qquad r=1,\ldots,4.
\end{equation*}
As shown in \cite{Olver-Pohjanpelto-05}, the Maurer--Cartan forms of the holomorphic pseudo-group $\G$ satisfy the same equations as the infinitesimal determining equations \eqref{eq: infinitesimal determining equations} under the correspondence
\[
z \to Z,\quad \oz \to \oZ,\quad u_r\to U_r,\quad v^r\to V^r,\quad\xi_A\to \mu_A,\quad
\oxi_A\to \omu_A,\quad \eta^r_A\to \alpha^r_A,\quad \phi^r_A\to \gamma^r_A.
\]
Thus we have the linear dependencies
\begin{equation}\label{MC-relations}
\aligned
\mu_{\oZ} &= \omu_Z = 0,\qquad \mu_{V^r}=\i\mu_{U_r},\qquad \omu_{V^r} = -\i\omu_{U_r},
\\
\gamma^r_Z &= -\i \alpha^r_Z, \qquad \gamma^r_{\oZ}=\i\alpha^r_{\oZ},\qquad
\gamma^r_{U_s} = - \alpha^r_{V^s},\qquad \gamma^r_{V^s} = \alpha^r_{U_s},
\\
\mu_{\oZ,X} &= 0,\qquad \mu_{ZV^r} = \i\mu_{ZU_r},\qquad \mu_{U_rV^s} = \i\mu_{U_rU_s},\qquad \mu_{V^rV^s} =-\mu_{U_rU_s},
\\
\omu_{Z,X} &= 0,\qquad \omu_{\oZ V^r} = -\i\omu_{\oZ U_r},\qquad \omu_{U_rV^s}=-\i\omu_{U_rU_s},\qquad \omu_{V^rV^s} = -\omu_{U_rU_s},
\\
\alpha^r_{Z\oZ}&=0,\qquad \alpha^r_{V^sV^t} = -\alpha^r_{U_sU_t},\qquad \alpha^r_{ZV^s} = \i \alpha^r_{ZU_s},\qquad \alpha^r_{\oZ V^s} = -\i \alpha^r_{\oZ U_s},
\\
\gamma^r_{Z,X} &= -\i\alpha^r_{Z,X},\qquad \gamma^r_{\oZ,X} = \i \alpha^r_{\oZ,X},\qquad
\gamma^r_{U_s,X} = - \alpha^r_{V^s,X},\qquad \gamma^r_{V^s,X} = \alpha^r_{U_s,X},
\endaligned
\end{equation}
for each $r,s,t=1,\ldots,4$, and $X\in\{Z, \oZ, U, V\}$. Therefore, a {\it basis} of our Maurer--Cartan forms is given by
\begin{equation}
\label{Basis-MC-original}
\mu_{Z^jU^\ell},\qquad \omu_{\oZ{}^j U^\ell}, \qquad \alpha^r_{Z^j U^\ell},\qquad \alpha^r_{\oZ{}^j U^\ell},\qquad \gamma^r_{U^\ell},
\end{equation}
where $j\in\mathbb N_0$, $\ell \in\mathbb N^4_0$ and $r =1, \ldots, 4$.

\subsection{Equivariant moving frames}

The equivariant moving frame theory, \cite{Olver-Fels-99,Olver-Pohjanpelto-05}, is a modern reformulation of Cartan's classical method of moving frames, \cite{Cartan-1935, Olver-1995}.  Of the most important results from this new perspective on moving frames is the introduction of the recurrence relations that characterize the algebra of differential invariants.  A key feature of these equations is that they can be derived without necessitating explicit formulas for either the moving frame or the invariants.  This has led to the introduction of what is called {\it symbolic invariant calculus}, \cite{M-2010}.   The coordinate implementation of the equivariant moving frame method can be found in the foundational papers \cite{Olver-Fels-99, Olver-Pohjanpelto-05}.  In this paper all computations are performed symbolically by exploiting the recurrence relations.

For $0\leq n \leq \infty$, consider the $n$-th order jet space ${\rm J}\n:={\rm J}\n(\mathbb C^5\simeq \mathbb R^{10}, 6)$ of $6$-dimensional real submanifolds in $\mathbb C^5$. Let $z, \oz, u=(u_1, u_2, u_3, u_4)$ denote the independent variables and $v=(v^1, v^2, v^3, v^4)$ the dependent ones.  Then, the submanifold jet coordinates are given by $\bz\n = (z, \oz, u, v\n)$, where $v\n =(\ldots\, v^r_{z^j \oz{}^k u^\ell}\,\ldots)$ collects the derivative coordinates of order $\leq n$ with $j, k\in \mathbb{N}_0$, $\ell \in \mathbb{N}_0^4$, and $r=1, \ldots, 4$.  The submanifold jet $\bz\ii$ at the base point $(z,\oz,u)$ provides the coefficients of the Taylor series expansion of the $6$-dimensional totally nondegenerate CR submanifolds $M \subset \mathbb C^5$ about $(z,\oz,u)$:
\[
v^r(z^*,\oz^*,u^*) = \sum_{j,k,\ell} \frac{v^r_{z^j \oz{}^k u^\ell}}{j! k! \ell!} (z^*-z)^j (\oz^*-\oz)^k (u^*-u)^\ell,\qquad r=1,\ldots,4.
\]

The biholomorphic pseudo-group $\G$ of $\mathbb{C}^5$ induces a prolonged action on ${\rm J}\n$
\begin{equation}\label{lifted invariants}
\bZ\n = (Z,\oZ,U,\,\ldots V^r_{Z^j\oZ{}^k U^\ell},\,\ldots) = \jet_n\varphi|_{\bz}\cdot \bz\n,\qquad \text{where}\qquad \varphi \in \G.
\end{equation}
In the following we use the multi-index notation $V^r_J$ to denote $V^r_{Z^j \oZ{}^k U^\ell}$ with $J=(j,k,\ell)\in \mathbb{N}_0^6$. Assuming $\ell = (l_1, \ldots, l_4)$, we also denote the {\it order} of $J$ by $|J|:=j+k+l_1+\ldots+l_4$.

The submanifold jet $\bz\n \in {\rm J}^{(n)}$ and the pseudo-group jet $\jet_n\varphi|_{\bz} \in \G\n$, with the source $\pi(\bz\n) = \bz=(z,\oz,u,v)$, locally parametrize the $n$-th order lifted bundle $\B\n$.  This bundle admits a groupoid structure\\
\begin{center}
\begin{tikzcd}
& \B\n \arrow[ld, "\bm{\sigma}\n" '] \arrow[rd, "\bm{\tau}\n"] & \\
{\rm J}^{(n)} & & {\rm J}^{(n)}
\end{tikzcd}
\end{center}
 where the source map $\bm{\sigma}\n(\bz\n,\jet_n\varphi|_{\bz}) = \bz\n$ is the projection onto the first component, and the target map is the prolonged action $\bm{\tau}\n(\bz\n,\jet_n\varphi|_{\bz}) = \bZ\n$.  There is a natural \emph{right action} of the holomorphic pseudo-group $\G$ on $\B\n$ given by combining the prolonged action \eqref{lifted invariants} together with the right composition of pseudo-group jets
\begin{equation}\label{eq: right action}
R_\psi(\bz\n,\jet_n\varphi|_{\bz}) = (\jet_n\psi|_{\bz}\cdot \bz\n,\jet_n(\varphi\circ \psi^{-1})|_{\psi(\bz)})\quad\text{for any}\quad \varphi,\,\psi\, \in\, \G
\end{equation}
where the composition $\varphi\circ \psi^{-1}$ is defined.  We note that the prolonged action \eqref{lifted invariants} is invariant under the right action \eqref{eq: right action}.  Thus the components of $\bZ\n$ are called {\it lifted invariants}.  The lifted invariants at the point $(Z, \oZ,U)$ provide the coefficients of the Taylor series expansion of the transformed submanifold $\varphi(M) \subset \mathbb{C}^5$:
\begin{equation}\label{normal form}
V^r(Z^*,\oZ{}^*,U^*) = \sum_{j,k,\ell} \frac{V^r_{Z^j \oZ{}^k U^\ell}}{j! k! \ell!} (Z^*-Z)^j (\oZ{}^*-\oZ)^k (U^*-U)^\ell,\qquad r=1,\ldots,4.
\end{equation}

\begin{Definition}
A \emph{(partial) right moving frame} of  order $n$ is a right-invariant local subbundle $\hB\n\subset \B\n$, meaning that $R_\psi(\hB\n) \subset \hB\n$ for all $\psi \in \G$ where the right action \eqref{eq: right action} is defined. If the subbundle $\hB\n$ forms the graph of a right-invariant section of $\B\n$, it defines an equivariant moving frame.
\end{Definition}

Let us denote by $\hrho\n\colon \hB\n \hookrightarrow \B\n$ the inclusion map induced by a (partial) right moving frame.  The construction of a (partial) moving frame depends on the choice of an appropriate cross-section to the prolonged pseudo-group orbits.  Assume, as it is the case with most applications, that the cross-section $\mathcal K$ is a coordinate cross-section specified by setting certain submanifold jet coordinates to constant values
\[
z = a,\qquad \oz = \overline{a},\qquad u = b,\qquad v^r_{z^j \oz{}^k u^\ell} = c^r_{j, k, \ell}.
\]
A (partial) moving frame is then obtained by solving the corresponding normalization equations
\begin{equation}\label{norm eq}
Z = a,\qquad \oZ = \overline{a},\qquad U = b,\qquad V^r_{Z^j \oZ{}^k U^\ell} = c^r_{j,k,\ell},
\end{equation}
for the pseudo-group jets.  If all the pseudo-group jets of order $\leq n$ can be solved in terms of the submanifold jet $\bz\n$, then one obtains a $n$-th order moving frame.  Otherwise, if certain pseudo-group jets remain unsolved for, then one obtains a {\it partial} moving frame, \cite{Valiquette-SIGMA}.  We note that the normalization equations \eqref{norm eq} specify the center of the Taylor series expansion in \eqref{normal form}, and the Taylor coefficients normalized to constant values which are referred as \emph{phantom invariants}.  The resulting Taylor series is called the \emph{normal form power series expansion} of the submanifold.

 Once a (partial) moving frame is constructed, it induces an {\it invariantization} projection that maps differential functions and, more generally, differential forms to their invariant counterparts on $\hB\n$.  Introducing the projection map $\pi_{\bm{\mu}}\colon \Omega^*(\hB\ii) \to \Omega^*(\hB\ii)$, which sets the Maurer--Cartan forms $\mu_A,
\omu_A,\alpha^j_A, \gamma^j_A$ to zero, the invariantization of a differential form $\Omega$ in  ${\rm J}\n$ is the invariant jet form
\[
\iota\n(\Omega ) = (\hrho\n)^* [\pi_{\bm{\mu}}((\bm{\tau}\n)^* \Omega )]
\]
defined on $\hB\n$. In particular, the invariantization of the submanifold jet coordinates yields the invariants $\iota\n(\bz\n) = (\hrho\n)^*(\bZ\n)$. To simplify the notation, we will not explicitly write the moving frame pull-back in the computations to come.  The context should make it clear that, for example, $V^r_J$ stands for $(\hrho\n)^*(V^r_J)$.  Also, we let $\iota=\iota\ii$.

Contact forms play an important role in the calculus of variations, \cite{KO-2003}, but they are not necessary for the problem considered in this paper.  Thus all computations in this paper are performed modulo contact forms.

\subsection{Recurrence formula}
\label{subsec-rec-rel}

We now introduce the {\it universal recurrence formula}, which computes the exterior derivative of an invariantized differential form, {\it symbolically}. First recall that the prolongation of the vector field $\bf v$ in \eqref{eq: v} is given by
\begin{equation}
\label{v-infty}
\vv^{(\infty)} = \xi \frac{\partial}{\partial z} + \oxi \frac{\partial}{\partial \oz} + \sum_{r=1}^4 \bigg(\eta^r \frac{\partial}{\partial u_r} + \sum_{J} \phi^{r;J}\frac{\partial}{\partial v^r_J}\bigg),
\end{equation}
where the vector components $\phi^{r;J}$ are defined recursively by the prolongation formula
\begin{subequations}
\label{prolong-formula}
\begin{equation}\label{eq: prolongation formula}
\phi^{r;J} = \Dt_J\bigg(\phi^r - \xi\, v^r_z - \oxi\, v^r_{\oz}  - \sum_{s=1}^4 \eta^s\, v^r_{u_s}\bigg) + \xi\, v^r_{J, z}  + \oxi\, v^r_{J, \oz}  + \sum_{s=1}^4 \eta^s\, v^r_{J, u_s},
\end{equation}
or equivalently
\begin{equation}\label{eq: prolongation formula 2}
\phi^{r; J,\sf x}=\Dt_{\sf x}\phi^{r;J}-(\Dt_{\sf x}\xi) \,v^r_{J,z}-(\Dt_{\sf x}\overline\xi) \,v^r_{J, \oz}- \sum_{s=1}^4 (\Dt_{\sf x}\eta^s)\, v^r_{J, u_s},
\end{equation}
\end{subequations}
where ${\sf x} \in \{z, \oz, u\}$ and $\Dt_{\sf x}$ denotes the total derivative operator with respect to the variable $\sf x$ and $\Dt_J = \Dt_z^j\Dt_{\oz}^k \Dt_{u_1}^{\ell_1}\cdots \Dt_{u_4}^{\ell_4}$ with $J=(j,k,\ell_1,\ldots,\ell_4) \in \mathbb{N}^6_0$.

\begin{Theorem} ({\rm cf. \cite[Theorem 25]{Olver-Pohjanpelto-08}})
{\it If $\Omega$ is a differential form on ${\rm J}^{(\infty)}$, then
\begin{equation}\label{eq: recurrence relation}
\dt [\iota(\Omega)]=\iota\big[\dt \Omega+{\bf v}^{(\infty)}(\Omega)\big],
\end{equation}
where ${\bf v}^{(\infty)}(\Omega)$ denotes the Lie derivative of $\Omega$ along ${\bf v}^{(\infty)}$.}
\end{Theorem}

To evaluate the correction term $\iota[{\bf v}^{(\infty)}(\Omega)]$ in the recurrence formula \eqref{eq: recurrence relation}, we extend the invariantization map to the vector field jets $\xi_A, \oxi_A, \eta^r_A,\phi^r_A$ by setting
\[
\iota(\xi_A) = (\hrho\ii)^*\mu_A,\qquad
\iota(\oxi_A) = (\hrho\ii)^*\omu_A,\qquad
\iota(\eta^r_A) = (\hrho\n)^*\alpha^r_A,\qquad
\iota(\phi^r_A) = (\hrho\n)^*\gamma^r_A,
\]
where $\mu_A, \omu_A, \alpha^r_A, \gamma^r_A$ are the Maurer--Cartan forms of the holomorphic pseudo-group $\G$. We denote, throughout this paper, the invariantization of the (horizontal) coframe by
\begin{equation}\label{eq: horizontal coframe}
\omega^Z = \iota(\dt z),\qquad
\omega^{\oZ} = \iota(\dt \oz),\qquad
\omega^r = \iota(\dt u_r),\qquad r=1,\ldots,4.
\end{equation}

Substituting for $\Omega$ in \eqref{eq: recurrence relation} the submanifold jet coordinates $\bz\n = (z, \oz, u, v\n)$, we obtain the recurrence relations
\begin{equation}\label{rec-formula}
\dt Z=\omega^Z+\mu, \qquad \dt \oZ=\omega^{\oZ}+\overline\mu,\qquad
\dt U_r=\omega^r+\alpha^r,\qquad
\dt V^r_J=\varpi^r_{J}+\iota(\phi^{r; J}),
\end{equation}
where
\[
\varpi^r_{J} =V^r_{J, Z}\,\omega^Z+V^r_{J, \oZ}\,\omega^{\oZ}+\sum_{s=1}^4V^r_{J, U_s}\,\omega^s,
\]
and $\iota(\phi^{r; J})$ is the invariantization of the prolonged vector field coefficient \eqref{eq: prolongation formula}.

A key observation that allows us to perform the computations symbolically is that the left hand side of the recurrence relations \eqref{rec-formula} for the phantom invariants \eqref{norm eq}, defining the normalization equations, is equal to zero.  These equations can then be symbolically solved for the normalized Maurer--Cartan forms \eqref{Basis-MC-original}.  Substituting the result in the remaining recurrence relations provides formulas for the exterior derivative of the (normalized) differential invariants.

\section{Total nondegeneracy and preliminary normalizations}
\label{sec-normal-form}

We now begin the construction of power series normal forms for six-dimensional totally nondegenerate CR manifolds $M\subset\mathbb C^5$.  To this end, we use order by order the recurrence relations \eqref{rec-formula} to determine which lifted invariants can be normalized. For brevity and thanks to the conjugation equality $V^r_{\overline J}=\overline{V^r_J}$, which follows from the fact that $M$ is a real surface, we will generally omit to write down the conjugate equations.

\subsection{Order zero}

At order zero, the recurrence relations are
\[
\dt Z = \omega^Z + \mu, \qquad
\dt \oZ = \omega^{\oZ} + \omu,\qquad
\dt U_r= \omega^r + \alpha^r, \qquad
\dt V^r = \varpi^r + \gamma^r,\qquad r=1,\ldots,4.
\]
If we set
\begin{equation}\label{order 0 normalizations}
Z = \oZ = U_r = V^r = 0,
\end{equation}
then the left hand side of the recurrence equations vanish, which allows us to solve them for the normalized Maurer--Cartan forms
\begin{equation*}
\mu=-\omega^Z, \qquad \omu = \omega^{\oZ},\qquad \alpha^r = -\omega^r, \qquad \gamma^r=-\varpi^r, \qquad r=1,\ldots,4.
\end{equation*}

\subsection{Order one}
The recurrence relations for the order one lifted invariants are
\begin{equation}
\label{eq: order 1 recurrence relations}
\dt V^r_Z = \varpi^r_Z -\i \alpha^r_Z + C^r_Z(V^{(1)},\bm{\mu}^{(1)}),\qquad
\dt V^r_{U_s} = \varpi^r_{U_s} + \gamma_{U_s}^r+ C^r_{U_s}(V^{(1)},\bm{\mu}^{(1)}),
\end{equation}
where  $C^r_Z$, $C^r_{U_s}$ are homogeneous linear functions in the first order Maurer--Cartan forms $\bm{\mu}^{(1)}$. By setting
\begin{equation*}
V^r_Z=V^r_{U_s}=0, \qquad r, s=1,\ldots,4,
\end{equation*}
the recurrence relations \eqref{eq: order 1 recurrence relations} become
\[
0 = \varpi^r_Z - \i \alpha^r_Z,\qquad 0 = \varpi^r_{U_s} + \gamma^r_{U_s},
\]
or equivalently
\[
\alpha^r_Z=-\i \varpi^r_Z, \qquad
 \text{and} \qquad \gamma^r_{U_s}= -\varpi^r_{U_s}.
\]
In light of the general prolongation formula \eqref{prolong-formula}, one obtains the following generalization of these already obtained normalizations.

\begin{Lemma}\label{lem-ord-1}
Let $j\in \mathbb{N}_0$, $\ell\in\mathbb N^4_0$ and $r = 1,\ldots,4$.  It is possible to normalize the Maurer--Cartan forms
 \renewcommand\labelenumi{\theenumi)}
\begin{enumerate}
 \item $\alpha^r_{Z^{j+1} U^\ell}$ by setting $V^r_{Z^{j+1} U^\ell}=0$,
 \item $\gamma^r_{U^\ell}$ by setting $V^r_{U^\ell}=0$.
\end{enumerate}
\end{Lemma}

\proof
For each $r=1,\ldots,4$, the component $\phi^{r;z}$ of the prolonged vector field \eqref{v-infty} is of the form
\[
\phi^{r;z}=1\phi^r_z+ \text{ terms  involving }  v^k_z, v^k_{\oz}, v^k_{u_l}, \qquad \text{for}\qquad k, l=1,\ldots,4.
\]
Using induction and the prolongation formula \eqref{eq: prolongation formula}, it follows for each $j\geq 0$ and $\ell\in\mathbb N^4_0$ that
\[
\phi^{r;z^ju^\ell}=1\phi^a_{z^ju^\ell}+\text{ terms  involving } v^k_J, \qquad \text{with} \qquad |J|\leq j+|\ell|.
\]
By the infinitesimal determining equations \eqref{eq: infinitesimal determining equations} we have $\phi^r_{z^ju^\ell}=-\i\eta^r_{z^ju^\ell}$, thus the recurrence relation of $V^r_{Z^jU^\ell}$ is
\[
\dt V^r_{Z^jU^\ell}=\varpi^r_{Z^jU^\ell}-\i\alpha^r_{Z^jU^\ell}+ \text{ terms  involving }  V^k_J, \qquad \text{with} \qquad |J|\leq j+|\ell|.
\]
Thanks to this equation, one can normalize the Maurer-Cartan form $\alpha^r_{Z^jU^\ell}$ by setting $V^r_{Z^jU^\ell}=0$ as was claimed in the first item of the lemma. The proof of the second item is completely similar and we leave it to the reader.
\endproof

\begin{Remark}
Lemma \ref{lem-ord-1} confirms that the {\it pluri-harmonic terms}, namely the coefficients of the monomials $z^j u^\ell$ and $\oz^j u^\ell$, can be removed from the normal form series expansion of our manifolds.
\end{Remark}

Let $M$ be a generic $6$-dimensional real submanifold of $\mathbb C^5$, represented locally as the graph of four real and real-analytic defining functions
\[
v^r=v^r(z,\oz,u_1,\ldots, u_4), \qquad r=1, \ldots, 4.
\]
After removing the pluri-harmonic terms, Mamai \cite{Mamai-13} shows in terms of the jet coordinates of $M$, that total nondegeneracy is equivalent to ask at the origin the following three circumstances
\begin{equation}\label{total-nondeg-conds}
 {\bf (i)}\, v^1_{z\oz}\neq 0, \qquad {\bf (ii)}\, {\sf det}\left(
                                \begin{array}{cc}
                                  v^2_{z^2\oz} & v^2_{z\oz^2}
                                  \vspace{.15cm}\\
                                  v^3_{z^2\oz} & v^3_{z\oz^2}\\
                                \end{array}
                              \right)\neq 0,
 \qquad {\bf (iii)} \, (v^4_{z^3\oz^2}, v^4_{z^2\oz^2})\neq (0,0).
\end{equation}
We will see shortly, in the course of the construction, how these circumstances manifest themselves in terms of the forthcoming (relative) invariants.
\subsection{Order two}

By Lemma \ref{lem-ord-1}, all the second order lifted invariants are normalized at this stage except for $V^1_{Z\oZ}$, $V^2_{Z\oZ}$, $V^3_{Z\oZ}$, $V^4_{Z\oZ}$. Taking into account the previous normalizations, we have the recurrence relations
\begin{equation}\label{order 2 req}
\dt V^r_{Z\oZ}=\varpi^r_{Z\oZ} - V^r_{Z\oZ}(\mu_Z + \omu_{\oZ}) + \sum_{s=1}^4 V^s_{Z\oZ}\,\alpha^r_{U_s},\qquad r=1,\ldots,4.
\end{equation}

In the most degenerate case, where the four lifted invariants $V^r_{Z\oZ}$ vanish locally around the origin, one can utilize the recurrence formula to prove by an induction on the order of lifted invariants that {\it all} $V^r_J$s with $|J|\geq 2$ vanish identically, as well. Then in this case, $M$ is nothing but the {\it flat} surface $\mathbb C\times\mathbb R^4$ defined by
\[
v^r=0, \qquad r=1, \ldots, 4.
\]
Thus, by the nondegeneracy assumption, at least one of the four already mentioned invariants, say $V^1_{Z\oZ}$, must be nonzero. This assumption coincides with the hypothesis {\bf (i)} in \eqref{total-nondeg-conds}. Consequently, in this branch we can set
\[
V^1_{Z\oZ} = 1,\qquad V^2_{Z\oZ} = V^3_{Z\oZ} = V^4_{Z\oZ} = 0.
\]
The recurrence relations \eqref{order 2 req} then yield the  normalizations
\[
\alpha^1_{U_1}= -\varpi^1_{Z\oZ}+\mu_Z+\omu_{\oZ}\qquad\text{and}\qquad
\alpha^r_{U_1}=-\varpi^r_{Z\oZ}, \qquad r=2, 3, 4.
\]
More generally, we have the following result.

\begin{Lemma}
\label{lem-ord-2}
The real Maurer--Cartan forms $\alpha^r_{U_1 U^\ell}$, with $r=1,\ldots,4$ and $\ell\in\mathbb N^4_0$, are normalized by setting $V^r_{Z\oZ U^\ell}= \delta^{r,\ell}_{1,(0,0,0,0)}$, where $\delta$ is the Kronecker delta.
\end{Lemma}

\begin{proof}
By the invariantization $\iota(\eta^r_{u_1u^\ell}) = \alpha^r_{U_1U^\ell}$, to show that the recurrence relations for the phantom invariants $V^r_{Z\oZ U^\ell} = 0$ can be solved for the Maurer--Cartan forms $\alpha^r_{U_1U^{\ell}}$, we must keep track of the vector field jets $\eta^r_{u_1u^{\ell}}$ in the prolonged vector field coefficients $\phi^{r;z\oz u^\ell}$.  From the prolongation formula \eqref{eq: prolongation formula 2} we have
\begin{equation}\label{v coef lemma}
\phi^{r;z\oz u^\ell}= \Dt_{u^\ell} \bigg(\phi^{r;z\oz}-\xi\, v^r_{zz\oz}-\oxi\,v^r_{z\oz\oz} - \sum_{s=1}^4 \eta^s\,v^r_{z\oz u_s}\bigg)
+ \xi\, v^r_{zz\oz u^\ell} + \oxi\, v^r_{z\oz\oz u^\ell} + \sum_{s=1}^4 \eta_s\, v^r_{z\oz u^{\ell} u_s}.
\end{equation}
In light of the infinitesimal determining equations \eqref{eq: infinitesimal determining equations}, the vector field jet $\eta^r_{u_1u^{\ell}}$ will appear via the derivatives of $\eta^r$ and $\phi^r$. We now keep track of the terms in \eqref{v coef lemma} that will not vanish in the recurrence relation of $V^r_{Z\oZ U^\ell}$ once we impose the order zero normalization \eqref{order 0 normalizations} and normalizations in Lemmas \ref{lem-ord-1} and \ref{lem-ord-2}.  Inspecting \eqref{v coef lemma}, it turns out that
\begin{equation}\label{phi-zozul}
\phi^{r;z\oz u^\ell} =\Dt_{u^\ell} \big(\phi^{r;z\oz}) - \sum_{K < \ell} \binom{\ell}{K}\big[\xi_{u^{\ell-K}}\, v^r_{zz\oz u^K} + \oxi_{u^{\ell-K}}\, v^r_{z\oz\oz u^K}\big]
 + \text{terms involving } v_{u^J}, v_{z\oz u^{J}},
\end{equation}
with
\[
\phi^{r;z\oz} = v^1_{z\oz}[ \eta^r_{u_1} - \delta^r_1(\xi_z+\oxi_{\oz})] + \text{terms involving } v^s_J\neq v^1_{zz}\text{ of order } |J| =1, 2,
\]
where we took into account the fact that $\phi^r_{z\oz} = 0$ using the infinitesimal determining equations \eqref{eq: infinitesimal determining equations}. Applying the differentiation $D_{u^\ell}$, one receives
\begin{equation}\label{eq: lemma2 prolongation term}
\Dt_{u^\ell}(\phi^{r;z\oz}) = v^1_{z\oz}[\eta^1_{u_1 u^\ell}-\delta^r_1(\xi_{z u^\ell} + \overline\xi_{\oz u^\ell})] + \ldots,
\end{equation}
where the unwritten terms $"\ldots"$ involve the jets $v^2_{z\oz}, v^3_{z\oz}, v^4_{z\oz}, v_{z\oz u^{J}}, v_{zu^J}, v_{\oz u^J}, v_{u^J}$, which all vanish once invariantized. Combining \eqref{phi-zozul} and \eqref{eq: lemma2 prolongation term}, and using the normalization $V^1_{Z\oZ} = 1$ with all other invariants vanished thus far, the recurrence relation for the phantom invariant $V^r_{Z\oZ U^\ell}=0$ reduces to
\[
0= \dt V^r_{Z\oZ U^\ell} = \varpi^r_{Z\oZ U^\ell}+\alpha^r_{U_1 U^\ell}-\delta^r_1(\mu_{Z U^\ell}+\overline\mu_{\oZ U^\ell}) - \sum_{K < \ell} \binom{\ell}{K}\big[V^r_{ZZ\oZ U^K} \mu_{U^{\ell-K}} + V^r_{Z\oZ\oZ U^K} \omu_{U^{\ell-K}}\, \big],
\]
which results in the claimed normalization
\[
\alpha^r_{U_1U^{\ell}} = -V^r_{ZZ\oZ U^\ell}\omega^Z - V^r_{Z\oZ\oZ U^\ell} \omega^{\oZ} + \delta^r_1(\mu_{Z U^\ell}+\overline\mu_{\oZ U^\ell}) + \sum_{K < \ell} \binom{\ell}{K}\big[V^r_{ZZ\oZ U^K} \mu_{U^{\ell-K}} + V^r_{Z\oZ\oZ U^K} \omu_{U^{\ell-K}}\, \big].
\]
of $\alpha^r_{U_1U^{\ell}}$.
\end{proof}

Owing to the length of the expressions in the forthcoming orders, we will occasionally write $"\equiv"$ instead of $"="$ when working modulo the horizontal coframe \eqref{eq: horizontal coframe}.

\subsection{Order three}

Taking into account Lemmas \ref{lem-ord-1} and \ref{lem-ord-2}, the order three recurrence relations that remain to be considered are
\begin{equation}\label{rec-3}
\aligned
\dt V^1_{Z^2\oZ} &\equiv  -V^1_{Z^2\oZ}\,\mu_Z + V^2_{Z^2\oZ}\,\alpha^1_{U_2} + V^3_{Z^2\oZ}\,\alpha^1_{U_3} + V^4_{Z^2\oZ}\,\alpha^1_{U_4} - \mu_{ZZ} + 4\i \omu_{U_1},\\
\dt V^k_{Z^2\oZ} &\equiv -V^k_{Z^2\oZ} \,(2\,\mu_Z+\omu_{\oZ}) +
V^2_{Z^2\oZ} \,\alpha^k_{U_2} + V^3_{Z^2\oZ}\,\alpha^k_{U_3} + V^4_{Z^2\oZ}\,\alpha^k_{U_4},\hspace{2cm} k=2,3,4.
\endaligned
\end{equation}
From the first equation, the Maurer--Cartan form $\omu_{U_1}$ can be readily normalized by setting
\[
V^1_{Z^2\oZ}=0.
\]

Next, for the other invariants and by means of some simple linear transformation, one easily annihilates in the defining power series of $M$ at least one of the coefficients $V^r_{Z^2\oZ}, r=2, 3, 4$, say: $V^4_{Z^2\oZ}$. Denote
\[
\Delta:=V^2_{Z^2\oZ}\cdot V^3_{Z\oZ^2}-V^2_{Z\oZ^2}\cdot V^3_{Z^2\oZ}.
\]
In light of the recurrence relations \eqref{rec-3}, one computes modulo horizontal coframe
\[
d\Delta\equiv\big(\alpha^2_{U_2}+\alpha^3_{U_3}-3\,\mu_Z-3\,\omu_{\oZ}\big)\,\Delta,
\]
which exhibits $\Delta$ as a {\it relative invariant}. Thus, its vanishing/nonvanishing remains invariant under holomorphic transformations. Taking into account the condition {\bf (ii)} of total nondegeneracy in \eqref{total-nondeg-conds}, we can locally assume here that $\Delta\neq 0$ which, according to its expression, plainly implies that $V^2_{Z^3\oZ}$ and $V^3_{Z^2\oZ}$ are locally nonzero. Let us set
\[
V^2_{Z^2\oZ}=1, \qquad V^3_{Z^2\oZ}=-\i,\qquad V^4_{Z^2\oZ}=0.
\]
Then the recurrence relations \eqref{rec-3} reduce to
\begin{equation}
\label{order3-red}
\aligned
0 &\equiv \alpha^1_{U_2} - \i \alpha^1_{U_3} - \mu_{ZZ} + 4\i \omu_{U_1},  \\
0 &\equiv \alpha^2_{U_2} - \i \alpha^2_{U_3} - (2\mu_Z + \omu_{\oZ}), \\
0 &\equiv \alpha^3_{U_2}-\i \alpha^3_{U_3} + \i(2\mu_Z + \omu_{\oZ}),\\
0 &\equiv \alpha^4_{U_2} -\i \alpha^4_{U_3}.
\endaligned
\end{equation}
By solving this system, one receives that
\[
\aligned
\omu_{U_1} &\equiv \frac{\i}{4}(\alpha^1_{U_2} - \i \alpha^1_{U_3} - \mu_{ZZ}),
\qquad
\mu_Z \equiv \frac{1}{3} \alpha^3_{U_3} - \i \alpha^2_{U_3},
\\
\alpha^2_{U_2} &\equiv \alpha^3_{U_3}, \qquad \alpha^3_{U_2} \equiv -\alpha^2_{U_3}, \qquad
\alpha^4_{U_2} \equiv \alpha^4_{U_3} \equiv 0.
\endaligned
\]

Upon further prolongations and with a similar arguments as the proofs of Lemmas \ref{lem-ord-1} and \ref{lem-ord-2}, we obtain the following general results.

\begin{Lemma}
\label{lem-ord-3}
For $j, k, l \in \mathbb{N}_0$ and $\ell\in\mathbb N^4_0$, it is possible to normalize
 \renewcommand\labelenumi{\theenumi)}
\begin{enumerate}
\item the Maurer--Cartan form $\mu_{U_1U^\ell}$ by setting $V^1_{Z^2\oZ U^\ell}=0$,

\item the Maurer--Cartan forms $\mu_{Z U^\ell}$ by setting $V^2_{Z^{2}\oZ U^\ell} = \delta^{\ell}_{(0,0,0,0)}$,

\item the two real Maurer--Cartan forms $\alpha^2_{U_2^{j+1} U_3^k U_4^l}$ and $\alpha^3_{U_2^{j+1} U_3^k U_4^l}$ by setting $V^3_{Z^{2}\oZ U_2^j U_3^k U_4^l}= -\i \delta^{j,k,l}_{0,0,0}$,

\item the real Maurer--Cartan forms $\alpha^4_{U_2^{j+1} U_3^{k} U_4^l}$ by setting ${\rm Re}V^4_{Z^2\oZ U_2^j U_3^{k} U_4^l}=0$,

\item the real Maurer--Cartan forms $\alpha^4_{U_3^{k+1} U_4^l}$ by setting ${\rm Im}V^4_{Z^2\oZ U_3^{k} U_4^l}=0$.
\end{enumerate}
\end{Lemma}

\begin{Remark}
Revisiting the recurrence relations \eqref{order3-red}, it was also possible to normalize the Maurer--Cartan forms $\alpha^k_{U_2} - \i \alpha^k_{U_3}$, $k=1,\ldots,4$.  In Lemma \ref{lem-ord-3} we avoid doing so as one of our goals is to normalize in each order as many lifted invariants as possible to construct a {\it minimal} moving frame.
\end{Remark}

\subsection{Order four}

We first consider the recurrence relations
\begin{equation*}
\begin{aligned}
\dt V^1_{Z^3\oZ} &\equiv 6\i\omu_{U_2}+6\,\omu_{U_3}-\mu_{Z^3} +V^3_{Z^3\oZ}\,\alpha^1_{U_3} + \cdots,\\
\dt V^1_{Z^2\oZ{}^2} &\equiv -24\,\re \mu_{U_3}+8\,\im \mu_{U_2}- 3\,V^2_{Z^2\oZ{}^2}\,\alpha^1_{U_2}+V^3_{Z^2\oZ{}^2}\,\alpha^1_{U_3}+ \cdots,\\
\dt V^2_{Z^2\oZ{}^2} &\equiv -2\,\alpha^1_{U_2}+\cdots,\\
\dt V^2_{Z^3\oZ} &\equiv -\frac{3}{2} \, \mu_{ZZ} -\frac{3}{2} \, \alpha^1_{U_2}+\cdots,\\
\dt V^3_{Z^2\oZ{}^2} &\equiv -2\,\alpha^1_{U_3}+\cdots,\\
\dt V^3_{Z^2\oZ U_1} &\equiv 2 \, \omu_{U_3}+2\i \omu_{U_2} +\frac{1}{2}\,(V^2_{Z^2\oZ{}^2}-\i V^3_{Z^2\oZ{}^2} + \i V^3_{Z^3\oZ})\,\alpha^1_{U_2}\\
&-\frac{1}{2}\,(\i V^2_{Z^2\oZ^2} + V^3_{Z^2\oZ{}^2}+V^3_{Z^3\oZ})\,\alpha^1_{U_3}+\cdots,
\end{aligned}
\end{equation*}
where the omitted terms are certain linear combinations of the Maurer--Cartan forms $\alpha^1_{U_4}$, $\alpha^2_{U_3}$, $\alpha^2_{U_4}$, $\alpha^3_{U_3}$, $\alpha^3_{U_4}$, $\alpha^4_{U_4}$. In light of these relations, we see it is possible to set
\[
V^1_{Z^3\oZ}=V^1_{Z^2\oZ{}^2}=V^2_{Z^2\oZ{}^2}=V^2_{Z^3\oZ}=V^3_{Z^2\oZ{}^2}=V^3_{Z^2\oZ U_1} = 0
\]
and solve their corresponding recurrence relations for the Maurer--Cartan forms
 \[
\mu_{Z^3}, \qquad \im \mu_{U_2}, \qquad \alpha^1_{U_2}, \qquad \mu_{ZZ},\qquad \alpha^1_{U_3},\qquad \omu_{U_3},\qquad \alpha^4_{U_3U_3}.
 \]
After prolongation, we obtain the following general result.

 \begin{Lemma}
 \label{lem-ord-4}
 For each $j, k, l \in \mathbb{N}_0$ and $\ell\in\mathbb N^4_0$, it is possible to normalize the Maurer--Cartan forms
 \renewcommand\labelenumi{\theenumi)}
\begin{enumerate}
\item $\mu_{Z^{j+3}U^\ell}$ by setting $V^1_{Z^{j+3}\oZ U^\ell}=0$,

\item ${\rm Re}\,\mu_{U_3^{k+1} U_4^l}$ by setting $V^1_{Z^2\oZ{}^2 U_3^k U_4^l}=0$,

\item  $\alpha^1_{U_2^{j+1} U_3^k U_4^l}$ by setting $V^2_{Z^2\oZ{}^2 U_2^j U_3^k U_4^l}=0$,

\item $\mu_{Z^2 U^\ell}$ by setting $V^2_{Z^3\oZ U^\ell} = 0$,

\item $\alpha^1_{U_3^{k+1} U_4^l}$ by setting $V^3_{Z^2\oZ{}^2 U_3^k U_4^l}=0$,

\item $\omu_{U_2^{j+1} U_3^{k}U_4^l}$ by setting $V^3_{Z^2\oZ U_1U_2^j U_3^k U_4^l}=0$.
\end{enumerate}
\end{Lemma}

Combining the order zero normalizations \eqref{order 0 normalizations} and Lemmas \ref{lem-ord-1}, \ref{lem-ord-2}, \ref{lem-ord-3}, \ref{lem-ord-4}, at this stage of the computations our cross-section consists of $Z=\oZ=U_r=0$ together with
\begin{equation}\label{partial-cross-sec}
\begin{gathered}
V^r_{Z\oZ U^\ell} = \delta^{r,\ell}_{1,(0,0,0,0)},\qquad
V^2_{Z^2\oZ U^\ell} = \delta^{\ell}_{(0,0,0,0)},\qquad
V^3_{Z^2\oZ U_2^j U_3^k U_4^l} = -\i \delta^{j,k,l}_{0,0,0}, \\
V^r_{Z^{j+1}U^\ell} = V^r_{U^\ell} = V^1_{Z^{j+2}\oZ U^\ell}
= V^1_{Z^2\oZ{}^2 U_2^j U_4^l}
= V^2_{Z^3\oZ U^\ell}= V^2_{Z^2\oZ{}^2 U_2^j U_3^k U_4^l}\\
=V^3_{Z^2\oZ U_1U_2^j U_3^k U_4^l}
= V^3_{Z^2\oZ{}^2U_3^k U_4^l}
=V^4_{Z^2\oZ U_3^{k} U_4^l} = \re V^4_{Z^2\oZ U_2^{j+1} U_3^k U_4^l}
 = 0,
\end{gathered}
\end{equation}
and their conjugations, where $j,k,l \in \mathbb{N}_0$, $\ell \in \mathbb{N}^4_0$, and $r=1,\ldots,4$. Furthermore, the Maurer--Cartan forms that remain to be normalized are
\begin{equation}\label{MC-3}
\im \mu_{U_3^{k+1} U_4^l}, \qquad \mu_{U_4^l}, \qquad \omu_{U_4^l}, \qquad \alpha^1_{U_4^l}, \qquad \alpha^2_{U_3^k U_4^l}, \qquad \alpha^3_{U_3^k U_4^l}, \qquad \alpha^4_{U_4^l},
\end{equation}
with $k,l\geq 0$ and with nonzero order.

Taking into account the normalizations performed thus far, the recurrence relations for the remaining fourth order (partially normalized) lifted invariants are
\begin{align}
\dt V^3_{Z^3\oZ}\equiv&\; V^3_{Z^3\oZ}\,\big(3\i\alpha^2_{U_3}-\frac{1}{3} \,\alpha^3_{U_3}\big)+V^4_{Z^3\oZ}\,\big(\i\alpha^2_{U_4}+\alpha^3_{U_4}\big),\notag\\
\dt V^4_{Z^2\oZ{}^2}\equiv&\; V^4_{Z^2\oZ{}^2}\big(\alpha^4_{U_4}-\frac{4}{3} \, \alpha^3_{U_3}\big),\notag\\
\dt V^4_{Z^3\oZ}\equiv&\; V^4_{Z^3\oZ}  \big(2\i\alpha^2_{U_3}-\frac{4}{3} \,\alpha^3_{U_3}+\alpha^4_{U_4}\big),\notag\\
\dt V^4_{Z^2\oZ U_1}\equiv&\; \frac{\rm i}{6} \big(V^3_{Z^3\oZ} V^4_{Z^2\oZ{}^2}-V^3_{Z\oZ{}^3}V^4_{Z^3\oZ}+6\,V^4_{Z^2\oZ U_1}\big) \alpha^2_{U_3}\notag\\
&+\frac{\rm i}{12}\big(5\,V^4_{Z^2\oZ{}^2} V^4_{Z^3\oZ}-2\,V^4_{Z^3\oZ} V^4_{Z\oZ{}^3}-3\, V^4_{Z^2\oZ{}^2} V^4_{Z^2\oZ{}^2}\big) \alpha^2_{U_4}\notag\\
&-\frac{5}{3}\,V^4_{Z^2\oZ U_1}\,\alpha^3_{U_3} -\frac{1}{4}V^4_{Z^2\oZ{}^2} \big(V^4_{Z^2\oZ{}^2} + V^4_{Z^3\oZ}\big) \alpha^3_{U_4}+V^4_{Z^2\oZ U_1}\,\alpha^4_{U_4},\label{rec-ord-4}\\
\dt (\im V^4_{Z^2\oZ U_2}) \equiv& -\frac{1}{16} (V^4_{Z^2\oZ{}^2})^2 \alpha^1_{U_4} -\frac{1}{12} V^3_{Z^3\oZ} V^3_{Z\oZ^3} \,\im V^4_{Z^3\oZ} \,\alpha^2_{U_3}\notag\\
&+ \big(\big(\frac{3}{4}\,\im V^4_{Z\oZ{}^2U_1}+\frac{1}{6}\, \im (V^4_{Z^3\oZ{}} V^3_{Z\oZ^3})+\frac{1}{16}V^4_{Z^2\oZ{}^2}\,\im V^3_{Z^3\oZ}\big) V^4_{Z^2\oZ^2}
\notag\\
& \ \ \ \ \ -\frac{1}{12} \im ((V^4_{Z^3\oZ})^2V^3_{Z\oZ{}^3})-\frac{1}{2}\im (V^4_{Z^2\oZ U_1}V^4_{Z\oZ{}^3})\big)\alpha^2_{U_4}\notag\\
&+\big(\big(\frac{1}{16}V^4_{Z^2\oZ^2} \re V^3_{Z^3\oZ{}}-\frac{1}{8}\,\re (V^4_{Z^3\oZ{}}V^3_{Z\oZ^3})-\frac{3}{4}\, \re V^4_{Z^2\oZ U_1}\big) V^4_{Z^2\oZ^2}\notag\\
&\ \ \ \ \ +\frac{1}{2}\, \re\!(V^4_{Z^2\oZ U_1}V^4_{Z^3\oZ{}})\big)\alpha^3_{U_4} -2\,\im V^4_{Z^2\oZ U_2}\, \alpha^3_{U_3} + \im V^4_{Z^2\oZ U_2} \, \alpha^4_{U_4}.\notag
\end{align}
In particular, the recurrence relations for $V^4_{Z^2\oZ{}^2}$ and $V^4_{Z^3\oZ}$ show that these functions, which correspond to condition {\bf (iii)} of the total nondegeneracy constraints \eqref{total-nondeg-conds}, are {\it relative} differential invariants. Based on this condition, these two lifted invariants can not be identically zero and thus we are led to consider the two major branches

\begin{description}
  \item[{\bf Branch A}] $V^4_{Z^3\oZ}\neq 0$.
  \item[{\bf Branch B}] $V^4_{Z^3\oZ}=0$ and $V^4_{Z^2\oZ{}^2}\neq 0$.
\end{description}
In the following, we consider each branch separately.

\section{Branch A: $V^4_{Z^3\oZ}\neq 0$}\label{sec-Branch-A}

In light of the recurrence relation for $V^4_{Z^3\oZ}$ in \eqref{rec-ord-4}, remaining group freedom acts on the relative invariant $V^4_{Z^3\oZ}$ via dilation
\[
V^4_{Z^3\oZ} \mapsto (a+\i b)V^4_{Z^3\oZ}\quad\text{where}\quad a\in \mathbb{R},\, b\geq 0.
\]
Since the imaginary part $b$ of the above dilation is always nonnegative, it is possible to set the complex invariant $V^4_{Z^3\oZ}$ either to $+\i$ or to $-\i$. Nevertheless, under the permutation $Z\leftrightarrow \oZ$, which preserves all the normalizations done thus far, we can restrict our attention to the case where
\begin{equation*}
V^4_{Z^3\oZ}= \i.
\end{equation*}
The recurrence relation for $V^4_{Z^3\oZ}$ then becomes
\begin{equation*}
0 \equiv 2\i \alpha^2_{U_3} - \frac{4}{3} \alpha^3_{U_3} +  \alpha^4_{U_4}
\end{equation*}
where, together with its conjugate, enables us to normalize the two real Maurer--Cartan forms
\begin{equation}\label{branchA mc norm}
\alpha^2_{U_3} \equiv 0,\qquad \alpha^3_{U_3} \equiv \frac{3}{4} \alpha^4_{U_4}.
\end{equation}
More generally we have the following result.

\begin{Lemma}
\label{lem-ord-4-Branch-A-1}
Let $k, l \in \mathbb{N}_0$.  The real Maurer--Cartan forms $\alpha^2_{U_3^{k+1} U_4^l}$ and $\alpha^3_{U_3^{k+1}U_4^{l}}$ can be normalized by setting $V^4_{Z^3\oZ U_3^k U_4^l}=\i \delta^{k,l}_{0,0}$.
\end{Lemma}

Substituting the normalizations \eqref{branchA mc norm} into the recurrence relation for $V^4_{Z^2\oZ{}^2}$ in \eqref{rec-ord-4} yields,  modulo the horizontal coframe
\[
\dt V^4_{Z^2\oZ{}^2}\equiv 0.
\]
This implies that the differential invariant $V^4_{Z^2\oZ{}^2}$ is now {\it independent of the pseudo-group parameters}. In the following, let us denote\footnote{It is just for some technical reason that we multiply $V^4_{Z^2\oZ{}^2}$ by $\frac{1}{4}$.} $\bm{b}:=\frac{1}{4}V^4_{Z^2\oZ{}^2}$.

On the other hand, the first recurrence relation in \eqref{rec-ord-4} becomes
\[
\dt V^3_{Z^3\oZ} \equiv -\frac{1}{4}\,V^3_{Z^3\oZ}\, \alpha^4_{U_4} - \alpha^2_{U_4} + \i \alpha^3_{U_4}.
\]
Thus, setting $V^3_{Z^3\oZ}=0$, it lets us to normalize the two real Maurer--Cartan forms $\alpha^2_{U_4}$ and $\alpha^3_{U_4}$.  More generally, we have the following result.

\begin{Lemma}
\label{lem-ord-4-Branch-A-2}
For $l \in \mathbb{N}_0$, it is possible to normalize the two real Maurer--Cartan forms $\alpha^2_{U_4^{l+1}}$ and $\alpha^3_{U_4^{l+1}}$ by setting $V^3_{Z^3\oZ U_4^l}=0$.
\end{Lemma}

At this stage of the problem, the remaining unnormalized Maurer--Cartan forms \eqref{MC-3} are
\begin{equation}\label{MC-4-B-A}
\im \mu_{U_3^{k+1} U_4^l}, \qquad \mu_{U_4^{l+1}}, \qquad \omu_{U_4^{l+1}}, \qquad \alpha^1_{U_4^{l+1}}, \qquad \alpha^4_{U_4^{l+1}},
\end{equation}
with $j, l \in \mathbb{N}_0$.
By the above normalizations applied so far, the two recurrence equations that remain to be considered in \eqref{rec-ord-4} are
\begin{subequations}
\label{rec-rel-ord-4-2}
\begin{align}
\dt(\im V^4_{Z^2\oZ U_2})&\equiv -\frac{1}{2}\, \im V^4_{Z^2\oZ U_2}\, \alpha^4_{U_4}-\bb^2\,\alpha^1_{U_4},\label{rec-rel-ord-4-2-eq1}
\\
\dt V^4_{Z^2\oZ U_1} &\equiv -\frac{1}{4}\, V^4_{Z^2\oZ U_1}\, \alpha^4_{U_4}.\label{rec-rel-ord-4-2-eq2}
\end{align}
\end{subequations}
Notice that if $\bb=\frac{1}{4} V^4_{Z^2\oZ^2}\neq 0$, then the recurrence relation \eqref{rec-rel-ord-4-2-eq1} allows us to normalize $\alpha^1_{U_4}$.  When $\bb =0$, we will see in Lemma \ref{lem-alpha1-U4-A'} that $\alpha^1_{U_4}$ can still be normalized at the next order $5$. With the aim of constructing a minimal moving frame, we therefore divide the subsequent computations according to  whether $\bb$ vanishes or not.  It refines {\bf Branch A} into
\begin{description}
\item[Branch A$^\prime$] $\bb \neq 0$,
\item[Branch A$^{\prime\prime}$] $\bb = 0$.
\end{description}

Before falling into these subbranches in details, we examine several fifth order recurrence relations that will facilitate for further normalization of the Maurer--Cartan forms. First, we have
\begin{equation}
\label{B-A-5-1}
\aligned
\dt V^3_{Z\oZ{}^3 U_1} &\equiv -\frac{3}{4}\,V^3_{Z\oZ{}^3 U_1}\,\alpha^4_{U_4}-2\,\mu_{U_4},
\\
\dt V^2_{Z^2\oZ{}^3}&\equiv -\frac{1}{2}\,V^2_{Z^2\oZ{}^3}\,\alpha^4_{U_4}+4\i \im \mu_{U_3}+(\i-\frac{15}{8}\,V^4_{Z^2\oZ{}^2}) \,\alpha^1_{U_4}.
\endaligned
\end{equation}
Thus, setting
\[
V^3_{Z\oZ{}^3 U_1} = \im V^2_{Z^2\oZ{}^3}=0,
\]
allows for the normalization of the Maurer--Cartan forms $\mu_{U_4}$ and $\im \mu_{U_3}$. In light of these normalizations, we then have the equations
\begin{equation*}
\aligned
\dt V^1_{Z^2\oZ{}^2 U_2} &\equiv - \frac{5}{4} V^1_{Z^2\oZ{}^2U_2} \alpha^4_{U_4}-32\, \im \mu_{U_3^2} + 21\, \im V^4_{Z^2\oZ{} U_1} \alpha^4_{U_4^2}+{\sf C_1} \,\alpha^1_{U_4}, \\
\dt V^3_{Z^3\oZ U_2} &\equiv -V^3_{Z^3\oZ U_2}\alpha^4_{U_4} - \frac{3}{4} \alpha^4_{U_4^2}+{\sf C_2} \,\alpha^1_{U_4},
\endaligned
\end{equation*}
for some certain polynomial coefficients $\sf C_1$ and $\sf C_2$ in terms of the lifted invariants $V^r_J$. These two relations allow us for the normalization of the Maurer--Cartan forms\footnote{Motivated by the second relation in \eqref{B-A-5-1}, one might propose to normalize $\im \mu_{U_3^2}$ by means of the sixth order recurrence relation of $dV^2_{Z^2\oZ^3U_3}$. But, seeking the minimality, we have succeeded in normalizing it already in the current order five.} $\im \mu_{U_3^2}$ and $\alpha^4_{U_4^2}$ by setting $V^1_{Z^2\oZ{}^2 U_2} = \re V^3_{Z^3\oZ U_2} = 0$. Upon further prolongation, we obtain the following result.

\begin{Lemma}\label{lem-A-1-ord-5}
For $j, l \in \mathbb{N}_0$, it is possible to normalize the Maurer--Cartan form
\renewcommand\labelenumi{\theenumi)}
\begin{enumerate}
\item $\mu_{U_4^{l+1}}$ by setting $V^3_{Z\oZ{}^3U_1U_4^l}=0$,

\item $\im \mu_{U_3 U_4^l}$ by setting $\im V^2_{Z^2\oZ{}^3U_4^l}=0$,

\item $\im \mu_{U_3^{j+2} U_4^l}$ by setting $V^1_{Z^2\oZ{}^2U_2 U_3^j U_4^l}=0$,

\item $\alpha^4_{U_4^{l+2}}$ by setting $\re V^3_{Z^3\oZ U_2 U_4^l} = 0$.
\end{enumerate}
\end{Lemma}

By virtue of this lemma, the list of the remaining unnormalized Maurer-Cartan forms is now reduced shortly to
\begin{equation*}
\alpha^1_{U_4^{l+1}}\quad \text{for}\quad  l\geq 0 \qquad {\rm and} \qquad \alpha^4_{U_4}.
\end{equation*}

Now, we continue the next normalizations through the two distinct branches {\bf A$^\prime$} and {\bf A$^{\prime\prime}$}.

\subsection{Branch A$^\prime$}

As already mentioned, in this branch one can solve the recurrence relation \eqref{rec-rel-ord-4-2-eq1} for normalizing $\alpha^1_{U_4}$, after setting  $\im V^4_{Z^2\oZ U_2}=0$. More generally,

 \begin{Lemma}
 \label{lem-Branch-A-alpha1-U4}
 For every $l\geq 0$, one can normalize the Maurer--Cartan form $\alpha^1_{U_4^{l+1}}$ by setting $\im V^4_{Z^2\oZ U_2U_4^l}=0$.
 \end{Lemma}

At this point, $\alpha^4_{U_4}$ is the only Maurer--Cartan form which is not yet normalized. We also emphasize that we already finished constructing the {\it complete} normal form of Branch A$^\prime$, introduced in the main Theorem \ref{main-result}. 

Meanwhile, the yet unused equation \eqref{rec-rel-ord-4-2-eq2} exhibits $V^4_{Z^2\oZ U_1}$ now as a relative differential invariant. Then, we have to build two subbranches in Branch {\bf A$^\prime$}:

\begin{description}
  \item[Branch A$^\prime$-1] $V^4_{Z^2\oZ U_1}\neq 0$,
  \item[Branch A$^\prime$-2] $V^4_{Z^2\oZ U_1} = 0$.
\end{description}

Let us proceed along these two subbranches.

\subsection{Branch A$^\prime$-1}

We now assume that $V^4_{Z^2\oZ U_1}$ is nonzero around the origin point. Then, in this case
\[
\|V^4_{Z^2\oZ U_1}\| = \sqrt{V^4_{Z^2\oZ U_1}V^4_{Z \oZ{}^2 U_1}} \neq 0.
\]
From the recurrence relation \eqref{rec-rel-ord-4-2-eq2} and its complex conjugate, we deduce also that
\[
\dt(\|V^4_{Z^2\oZ U_1}\|) \equiv -\frac{1}{8} \|V^4_{Z^2\oZ U_1}\| \,\alpha^4_{U_4},
\]
which allows to normalize $\alpha^4_{U_4}$ by setting $\|V^4_{Z^2\oZ U_1}\| =1$.

Now, we have normalized all the Maurer--Cartan forms of the holomorphic pseudo-group $\G$. It results the uniqueness --- up to a discrete isotropy group --- of the already constructed normal form in Branch {\bf A$^\prime$-1}.

\begin{Theorem}\label{Theorem-A-1}
Let $M\subset\mathbb C^4$ be a $6$-dimensional totally nondegenerate manifold belonging to Branch {\bf A$^\prime$-1}. Then there exists a transformation of $\mathbb{C}^4$ locally mapping $M$ to the normal form\footnote{In the normal form power series \eqref{normal form}, we set $V^r=v^r$, $Z^*=z$, $\oZ^*=\oz$, and $U^*=u$ to match the notation of the defining equations \eqref{def-eq}.}
\begin{align}
v^1&=z\oz+\sum_{j+k+|\ell|\geq 5} \frac{V^1_{Z^j\oZ{}^k U^\ell}}{j! k! \ell!} z^j\oz{}^k u^\ell,\nonumber
\\
v^2&=\frac{1}{2}\,(z^2\oz+z\oz{}^2)+\sum_{j+k+|\ell | \geq 5} \frac{V^2_{Z^j\oZ{}^k U^\ell}}{j! k! \ell!} z^j\oz{}^k u^\ell, \label{NM-Branch-A-1-NF}
\\
v^3&=-\frac{\rm i}{2}\,(z^2\oz-z\oz^2)+\sum_{j+k+|\ell |\geq 5} \frac{V^3_{Z^j\oZ{}^k U^\ell}}{j! k! \ell!} z^j\oz{}^k u^\ell,\nonumber
\\
v^4&=\frac{\rm i}{6} \, (z^3\oz- z\oz{}^3)+\bb \,z^2\oz{}^2 +\frac{\exp(\i \theta)}{2} z^2\oz u_1+\frac{\exp(-\i\theta)}{2} z\oz{}^2 u_1\nonumber\\
&\hspace{4.5cm}
+\sum_{j+k+|\ell |\geq 5} \frac{V^4_{Z^j\oZ{}^k U^\ell}}{j! k! \ell!} z^j\oz{}^k u^\ell,\nonumber
\end{align}
where $\theta:=\arg(V^4_{Z^2\oZ U_1})$ and the coefficients $V_J$ of order $|J| \geq 5$ satisfy
\eqref{partial-cross-sec} together with
\begin{multline}
\label{cross-sec-A-1-+}
0= V^4_{Z^3\oZ U_3^k U_4^l} = V^3_{Z^3\oZ U_4^l} = V^3_{Z^3\oZ{}U_1U_4^l}=\im V^2_{Z^3\oZ{}^2 U_4^l}  \\
= \im V^4_{Z^2\oZ U_2 U_4^l} = V^1_{Z^2\oZ{}^2 U_2 U_3^j U_4^l} = \re V^3_{Z^3\oZ U_2 U_4^l},
\end{multline}
for $j, l \geq 0$.
\end{Theorem}

\subsection{Branch A$^\prime$-2}

We now assume that $V^4_{Z^2\oZ U_1} = 0$. In this case, the recurrence equation \eqref{rec-rel-ord-4-2-eq2} of this invariant no longer provides a direct normalization. However, it results in a homogeneous system formed by its coefficients of the independent horizontal forms $\omega^Z, \omega^{\oZ}, \omega^1, \ldots, \omega^4$. The solution of this system expresses $V^4_{Z^3\oZ U_1}$, $V^4_{Z^2\oZ^2 U_1}$ and $V^4_{Z^2\oZ U_1U_j}$ in terms of other lifted invariants of the same order. Among these relations, we have in particular that
\begin{align}
V^4_{Z^2\oZ U_1U_4} &= 0,\notag \\
V^4_{Z^2\oZ U_1 U_3} &= \frac{1}{12}\,(\i + 16\,\bb)\,V^3_{Z\oZ{}^3U_3} + \frac{1}{12}\,(\i + 4\,\bb + 48\i \bb^2) V^3_{Z^3\oZ U_3},\notag \\
V^4_{Z^2\oZ U_1U_2} &= (4\i \bb^2-\bb)\,V^3_{Z^3\oZ{} U_2} + \frac{1}{4}\,(\i + \bb) \,V^3_{Z^2\oZ{}^2U_2},\notag \\
V^4_{Z^2\oZ U_1^2} &= \frac{1}{4}\,(1+4\i \bb)\,V^2_{Z^2\oZ{}^2U_1}+\frac{1}{4}\,(\i + 4\,\bb)\,V^3_{Z^2\oZ{}^2U_1}.\label{A-1p constraints}
\end{align}

Since at this stage no additional fourth order lifted invariant is available for normalizing the Maurer-Cartan form $\alpha^4_{U_4}$, we must advance to order five where, we find the following set of unconsidered lifted differential invariants:
\begin{multline*}
\mathcal{R}_\textbf{A$^\prime$-2} = \big\{V^1_{Z^3\oZ{}^2}, V^1_{Z^2\oZ{}^2U_1}, V^2_{Z^4\oZ}, \re V^2_{Z^3\oZ{}^2}, V^2_{Z^2\oZ{}^2U_1}, V^3_{Z^4\oZ}, V^3_{Z^3\oZ{}^2}, \im V^3_{Z^3\oZ U_2}, V^3_{Z^3\oZ U_3}, V^3_{Z^2\oZ{}^2U_1}, \\
 V^3_{Z^2\oZ{}^2U_2}, V^3_{Z^2\oZ U_1^2}, V^4_{Z^4\oZ}, V^4_{Z^3\oZ{}^2}, V^4_{Z^3\oZ U_2}, \im V^4_{Z^2\oZ U_2U_2}, \im V^4_{Z^2\oZ U_2U_3}, V^4_{Z^2\oZ{}^2U_2}, V^4_{Z^2\oZ{}^2 U_3}, V^4_{Z^2\oZ{}^2 U_4}\big\},
\end{multline*}
 with the following recurrence relations
\begin{align}
&\dt V^1_{Z^3\oZ{}^2}\equiv -\frac{3}{4} V^1_{Z^3\oZ{}^2} \, \alpha^4_{U_4},&
\qquad&
\dt V^1_{Z^2\oZ{}^2 U_1}\equiv -V^1_{Z^2\oZ{}^2U_1}\,\alpha^4_{U_4}, \notag \\
&\dt V^2_{Z^4\oZ}\equiv -\frac{1}{2} V^2_{Z^4\oZ}\,\alpha^4_{U_4}, & &
\dt (\re V^2_{Z^3\oZ{}^2})\equiv-\frac{1}{2} \re V^2_{Z^3\oZ{}^2}\,\alpha^4_{U_4}, \notag \\
&\dt V^2_{Z^2\oZ{}^2U_1}\equiv -\frac{3}{4}\, V^2_{Z^2\oZ{}^2U_1} \,\alpha^4_{U_4},& &
\dt V^3_{Z^4\oZ}\equiv -\frac{1}{2}\,V^3_{Z^4\oZ}\,\alpha^4_{U_4}, \notag \\
&\dt V^3_{Z^3\oZ{}^2}\equiv -\frac{1}{2}\,V^3_{Z^3\oZ{}^2}\,\alpha^4_{U_4},& &
\dt(\im V^3_{Z^3\oZ U_2})\equiv - \im V^3_{Z^3\oZ U_2}\,\alpha^4_{U_4}, \label{Rec-rel-ord-5-A-2} \\
&\dt V^3_{Z^3\oZ U_3}\equiv -V^3_{Z^3\oZ U_3}\,\alpha^4_{U_4},& &
\dt V^3_{Z^2\oZ{}^2 U_1}\equiv -\frac{3}{4}\, V^3_{Z^2\oZ{}^2 U_1}\,\alpha^4_{U_4}, \notag \\
& \dt V^3_{Z^2\oZ{}^2 U_2}\equiv -V^3_{Z^2\oZ{}^2 U_2}\,\alpha^4_{U_4},& &
\dt V^3_{Z^2\oZ U_1^2}\equiv -V^3_{Z^2\oZ U_1^2}\,\alpha^4_{U_4}, \notag \\
&\dt V^4_{Z^4\oZ}\equiv -\frac{1}{4}\,V^4_{Z^4\oZ}\,\alpha^4_{U_4},& &
\dt V^4_{Z^3\oZ{}^2}\equiv -\frac{1}{4}\,V^4_{Z^3\oZ{}^2}\,\alpha^4_{U_4}, \notag \\
&\dt V^4_{Z^2\oZ{}^2 U_2}\equiv -\frac{3}{4}\,V^4_{Z^2\oZ{}^2 U_2}\,\alpha^4_{U_4},& &
\dt V^4_{Z^2\oZ{}^2 U_3}\equiv -\frac{3}{4}\,V^4_{Z^2\oZ{}^2 U_3}\,\alpha^4_{U_4}, \notag \\
&\dt V^4_{Z^2\oZ{}^2 U_4}\equiv -V^4_{Z^2\oZ{}^2 U_4}\,\alpha^4_{U_4}, & &
\dt V^4_{Z^3\oZ{} U_2}\equiv -\frac{3}{4}\,V^4_{Z^3\oZ{} U_2}\,\alpha^4_{U_4},\notag \\
&\dt (\im V^4_{Z^2\oZ{} U_2U_2})\equiv -\frac{5}{4}\,\im V^4_{Z^2\oZ{} U_2U_2}\,\alpha^4_{U_4}, & &
\dt (\im V^4_{Z^2\oZ{} U_2U_3})\equiv -\frac{5}{4}\,\im V^4_{Z^2\oZ{} U_2U_3}\,\alpha^4_{U_4}. \notag
\end{align}

In view of these relations, any {\it relative} invariant $R\in\mathcal R_\textbf{A$^\prime$-2}$ may serve to normalize the only remaining Maurer-Cartan form $\alpha^4_{U_4}$ provided $R\neq 0$. Then, we anticipate two further subbranches of Branch {\bf A$^\prime$-2} for the forthcoming normalizations:

\begin{description}
  \item[Branch A$^\prime$-2-1] If either of the lifted invariants in $\mathcal R_\textbf{A$^\prime$-2}$ is nonzero.
  \item[Branch A$^\prime$-2-2] If all members of $\mathcal R_\textbf{A$^\prime$-2}$ vanish, locally.
\end{description}

\subsection{Branch A$^\prime$-2-1}

Let $R\in\mathcal R_\textbf{A$^\prime$-2}$ be a nonzero lifted invariant. By the recurrence relation of $\| R \|^2=R\cdot\overline R$, it is easy to verify that $\| R \|$ is a nonzero relative invariant, as well. We set
\[
\| R \|=1
\]
and solve the corresponding recurrence relation $\dt (\|R\|)$ for normalizing $\alpha^4_{U_4}$. Then in this case, we receive a complete normal form which corresponds to a unique normal form.

\begin{Theorem}
\label{Theorem-A-2-1}
Let $M\subset\mathbb C^4$ be a $6$-dimensional totally nondegenerate manifold belonging to Branch {\bf A$^\prime$-2-1}. There exists an origin-preserving transformation of $\mathbb{C}^5$, mapping $M$ to the normal form
\begin{equation*}
\aligned
v^1&=z\oz+\sum_{j+k+|\ell |\geq 5} \frac{V^1_{Z^j\oZ{}^k U^\ell}}{j! k! \ell!} z^j\oz^k u^\ell,
\\
v^2&=\frac{1}{2}\,(z^2\oz+z\oz^2)+\sum_{j+k+|\ell |\geq 5} \frac{V^2_{Z^j\oZ{}^k U^\ell}}{j! k! \ell!} z^j\oz^k u^\ell,
\\
v^3&=-\frac{\rm i}{2}\,(z^2\oz-z\oz^2)+\sum_{j+k+|\ell |\geq 5} \frac{V^3_{Z^j\oZ{}^k U^\ell}}{j! k! \ell!} z^j\oz^k u^\ell,
\\
v^4&=\frac{\rm i}{6} \, (z^3\oz- z\oz^3)+\bb \,z^2\oz^2+\sum_{j+k+|\ell |\geq 5} \frac{V^4_{Z^j\oZ{}^k U^\ell}}{j! k! \ell!} z^j\oz^k u^\ell,
\endaligned
\end{equation*}
where the coefficients $V_J$ and their conjugations satisfy the normalizations \eqref{partial-cross-sec} and \eqref{cross-sec-A-1-+}, along with
\begin{gather*}
\| R \| = 1 \qquad \text{for some non-zero relative invariant $R\in \mathcal{R}_\textbf{A$^\prime$-2}$},
\end{gather*}
and supplemented with the relations \eqref{A-1p constraints} and their prolongations.  This normal form is unique up to a discrete group of transformations.
\end{Theorem}

\subsection{Branch A$^\prime$-2-2}
\label{sub-sec-A-p-2-2}

Now let us assume that all order five lifted differential invariants --- and in particular those belonging to the set $\mathcal R_\textbf{A$^\prime$-2}$ --- vanish, identically. In that case, we shall seek any possible normalization of $\alpha^4_{U_4}$ in the next order six. As expected, this order presents a plenty of yet unconsidered lifted invariants. However, surprisingly, solving the homogeneous system formed by the coefficients of the horizontal forms in the recurrence relations \eqref{Rec-rel-ord-5-A-2} reveals that except
\begin{equation}\label{rel-ord-6-A-2-2}
\aligned
&  V^1_{Z^2\oZ^2U_2U_2}, \qquad V^4_{Z^2\oZ U_1U_2U_2}, \qquad V^4_{Z^2\oZ U_1U_2U_3}, \qquad V^4_{Z^2\oZ U_1U_3U_3},
\\
& \im V^4_{Z^2\oZ U_2U_2U_2}, \qquad V^4_{Z^2\oZ^2U_1U_4}, \qquad \im V^4_{Z^3\oZ U_1U_4},
\endaligned
\end{equation}
all other sixth order lifted invariants vanish, identically. Among these seven lifted differential invariants and for some technical reason --- that we will find it soon --- we first examine the recurrence relation of $\im V^4_{Z^3\oZ U_1U_4}$:
\begin{equation*}
\aligned
\dt (\im V^4_{Z^3\oZ U_1U_4})&\equiv\frac{3}{2}\,\im V^4_{Z^3\oZ U_1U_4}\,\alpha^4_{U_4}.
\endaligned
\end{equation*}
This relation reveals that the lifted invariant $\im V^4_{Z^3\oZ U_1U_4}$ is relative and, accordingly, we have to build once more the following two subbranches

\begin{description}
  \item[Branch A$^\prime$-2-2-1] If the real lifted differential invariant $\im V^4_{Z^3\oZ U_1U_4}$ is nonzero.
  \item[Branch A$^\prime$-2-2-2] If $\im V^4_{Z^3\oZ U_1U_4} = 0$.
\end{description}

In the first subbranch {\bf A$^\prime$-2-2-1}, one clearly is permitted to set $\|\im V^4_{Z^3\oZ U_1U_4}\|=1$ and normalize the last Maurer-Cartan form $\alpha^4_{U_4}$.

\begin{Theorem}
\label{Theorem-A-2-2-1}
Let $M\subset\mathbb C^4$ be a $6$-dimensional totally nondegenerate manifold belonging to Branch {\bf A$^\prime$-2-2-1}. Then there exists an origin-preserving transformation of $\mathbb{C}^5$, mapping it to the normal form
\begin{equation*}
\begin{aligned}
v^1&=z\oz+\sum_{j+k+|\ell |\geq 6} \frac{V^1_{Z^j\oZ{}^k U^\ell}}{j! k! \ell!} z^j\oz^k u^\ell,
\\
v^2&=\frac{1}{2}\,(z^2\oz+z\oz^2)+\sum_{j+k+|\ell |\geq 7} \frac{V^2_{Z^j\oZ{}^k U^\ell}}{j! k! \ell!} z^j\oz^k u^\ell,
\\
v^3&=-\frac{\rm i}{2}\,(z^2\oz-z\oz^2)+\sum_{j+k+|\ell |\geq 7} \frac{V^3_{Z^j\oZ{}^k U^\ell}}{j! k! \ell!} z^j\oz^k u^\ell,
\\
v^4&=\frac{\rm i}{6} \, (z^3\oz- z\oz^3)+\bb \,z^2\oz^2+\sum_{j+k+|\ell |\geq 6} \frac{V^4_{Z^j\oZ{}^k U^\ell}}{j! k! \ell!} z^j\oz^k u^\ell,
\end{aligned}
\end{equation*}
where the coefficients $V_J$ and their conjugations satisfy the normalizations \eqref{partial-cross-sec} and \eqref{cross-sec-A-1-+} supplemented by $\|\im V^4_{Z^3\oZ U_1U_4}\|=1$. Moreover, all lifted invariants $V^1_J$ and $V^4_J$ with $|J|=6$ vanish except those presented in \eqref{rel-ord-6-A-2-2}. This normal form is unique up to a discrete transformations group.
\end{Theorem}

But in Branch {\bf A$^\prime$-2-2-2}, we show that the normalization of $\alpha^4_{U_4}$ is impossible, even at higher orders. For this purpose, we invoke Cartan's principles in his classical theory of equivalence problems. Before it, let us emphasize that following the vanishing all order five lifted invariants, we have derived $\dt \bb=0$ indicating that $\bb$ is now a {\it constant} function.

Granting the normalizations $0=V^j_Z=V^j_{\oZ}=V^j_{U_k}$ we set for $j, k=1,\ldots, 4$, the structure equations of the horizontal {\it lifted} coframe of $M$ are (see \cite[Theorem 2.2]{Valiquette-11} for the formula)
\begin{equation}\label{struc-equ}
\begin{aligned}
\dt\omega^Z &= \mu_Z\wedge\omega^Z+ \sum_{s=1}^4 \mu_{U_s}\wedge\omega^s,\qquad
\dt\omega^{\oZ}=\overline{d\omega^Z},
\\
\dt \omega^r &= \alpha^r_Z\wedge\omega^Z+\alpha^r_{\oZ}\wedge\omega^{\oZ}+ \sum_{s=1}^4 \alpha^r_{U_s}\wedge\omega^s, \qquad r=1,\ldots,4.
\end{aligned}
\end{equation}
By vanishing all the lifted invariants in order five along with those vanished in order six at the beginning of this subsection \ref{sub-sec-A-p-2-2}, the normalized Maurer--Cartan forms visible among \eqref{struc-equ} take the form
\begin{gather*}
\mu_Z=\frac{1}{4}\,\alpha^4_{U_4}, \qquad \alpha^1_{Z}=-{\rm i}\,\omega^{\oZ}, \qquad \alpha^1_{\oZ}={\rm i}\,\omega^{Z}, \qquad \alpha^1_{U_1}=\frac{1}{2}\,\alpha^4_{U_4},
\\
\alpha^2_{U_1}=-(\omega^Z+\omega^{\oZ}), \qquad \alpha^2_{U_2}=\frac{3}{4}\,\alpha^4_{U_4}, \qquad \alpha^3_{U_1}={\rm i}\,(\omega^Z-\omega^{\oZ}), \qquad \alpha^3_{U_3}=\frac{3}{4}\,\alpha^4_{U_4},
\\
\alpha^4_{U_2}=-\frac{1}{2}\,(4\,\bb+{\rm i})\,\omega^Z-\frac{1}{2}\,(4\,\bb-{\rm i})\,\omega^{\oZ}, \qquad
\alpha^4_{U_3}=\frac{1}{2}\,(1+4{\rm i}\,\bb)\,\omega^Z+\frac{1}{2}\,(1-4{\rm i}\,\bb)\,\omega^{\oZ}.
\end{gather*}
Inserting these expressions into the structure equations \eqref{struc-equ} gives rise to
\begin{equation}\label{struc-eq-A-2-2-2}
\aligned
\dt\omega^Z &= \frac{1}{4}\,\alpha^4_{U_4}\wedge\omega^Z, \qquad \qquad\qquad d\omega^{\oZ}=\frac{1}{4}\,\alpha^4_{U_4}\wedge\omega^{\oZ},
\\
\dt\omega^1 &= 2\i \omega^Z\wedge\omega^{\oZ}+\frac{1}{2}\,\alpha^4_{U_4}\wedge\omega^1,
\\
\dt \omega^2 &= -\omega^Z\wedge\omega^1-\omega^{\oZ}\wedge\omega^1+\frac{3}{4}\,\alpha^4_{U_4}\wedge\omega^2,
\\
\dt \omega^3 &= \i \omega^Z\wedge\omega^1 - \i\omega^{\oZ}\wedge\omega^1+\frac{3}{4}\,\alpha^4_{U_4}\wedge\omega^3,
\\
\dt \omega^4 &= -\frac{1}{2}\big(4\,\bb+\i\big)\,\omega^Z\wedge\omega^2 - \frac{1}{2}\big(4\,\bb-\i \big)\,\omega^{\oZ}\wedge\omega^2
\\
&\hspace{0.375cm}+\frac{1}{2}\big(1+4\i \bb\big)\,\omega^Z\wedge\omega^3 + \frac{1}{2}\big(1-4\i \bb\big)\,\omega^{\oZ}\wedge\omega^3 + \alpha^4_{U_4}\wedge\omega^4.
\endaligned
\end{equation}
These structure equations are of {\it constant type} and thus, no further normalizations are available on them. Accordingly, the equivalence problem to the manifolds $M$ in the current branch is identified by the above structure equations defined on the $7$-dimensional prolonged manifold $M\times\mathcal G^{\sf red}$ where $\mathcal G^{\sf red}$ denotes the $1$-dimensional reduced subgroup of $\mathcal G$ with the associated Maurer--Cartan form $\alpha^4_{U_4}$. Another application of the formula introduced in \cite[Theorem 2.2]{Valiquette-11} gives that
\[
\dt \alpha^4_{U_4}=\omega^Z\wedge\alpha^4_{ZU_4}+\omega^{\oZ}\wedge\alpha^4_{\oZ U_4}
+\alpha^4_Z\wedge\mu_{U_4} +\alpha^4_{\oZ}\wedge\overline\mu_{U_4}
+\sum_{s=1}^4 \bigg(\omega^s\wedge\alpha^4_{U_sU_4}
+\alpha^4_{U_s}\wedge\alpha^s_{U_4}-\alpha^4_{V^s}\wedge\alpha^s_{V^4}\bigg),
\]
where, after applying necessary computations, extensively simplifies to
\begin{equation}\label{dalpha4-U4}
\dt \alpha^4_{U_4}=0.
\end{equation}
It follows that the equivalence problem to manifolds $M$ belonging to Branch {\bf A$^\prime$-2-2-2} is encoded by the constant type structure equations \eqref{struc-eq-A-2-2-2}-\eqref{dalpha4-U4} on the prolonged space $M\times\mathcal G^{\sf red}$. This implies that two manifolds $M$ and $M'$ of this branch possess the same coefficient $\bb$ in their defining equations are holomorphically equivalent. As we know, Beloshapka's model surfaces $M(\frac{\rm i}{6}, \bb)$, $\bb\neq 0$, with the defining functions (cf. \eqref{def-eq-model})
\begin{equation}\label{model-A-2-2-2}
\aligned
v^1&=z\oz,
\\
v^2&=\frac{1}{2}\,(z^2\oz+z\oz^2),
\\
v^3&=-\frac{\rm i}{2}\,(z^2\oz-z\oz^2),
\\
v^4&=\frac{\rm i}{6}\, (z^3\oz- z\oz^3)+\bb \,z^2\oz^2,
\endaligned
\end{equation}
belong to this branch. Thus, we have
\begin{Theorem}
\label{Th-A-2-2-2}
For each $6$-dimensional totally nondegenerate manifold $M$ belonging to Branch {\bf A$^\prime$-2-2-2}, there exists some origin-preserving holomorphic transformation, mapping it to a model surface defined by \eqref{model-A-2-2-2}. The isotropy group of such surfaces is $1$-dimensional and the already mentioned transformation is unique up to the action of this isotropy group. Moreover, two manifolds belonging to this branch are equivalent if and only if their associated normal forms \eqref{model-A-2-2-2} admit the same structure equations \eqref{struc-eq-A-2-2-2}.
\end{Theorem}

\begin{Remark}
Setting $-\bb$ instead of $\bb$ in \eqref{struc-eq-A-2-2-2}, one finds that the appearing structure equations of $M(\frac{\rm i}{6}, -\bb)$ turn back to those of $M(\frac{\rm i}{6}, \bb)$ once we apply the substitutions of the horizontal forms
\[
\omega^Z\leftrightarrow\omega^{\oZ}, \qquad \omega^1\rightarrow -\omega^1, \qquad \omega^2\rightarrow -\omega^2.
\]
Consequently, for each real integer $\bb$, the two model surfaces $M(\frac{\rm i}{6}, \bb)$ and $M(\frac{\rm i}{6}, -\bb)$ are biholomorphically equivalent. This result is also confirmed by the explicit expression of the real invariant $\mathfrak J$ introduced in \eqref{J}. Thus,  we can assume here $\bb> 0$.
\end{Remark}

According to \cite{Valiquette-SIGMA}, the isotropy group at the origin of the model surfaces \eqref{model-A-2-2-2} is $\mathcal G^{\sf red}$, identified with the single Maurer--Cartan form $\alpha^4_{U_4}$. The infinitesimal counterpart of this group is generated by the real parts of the holomorphic vector fields $A(z,w)\partial_z+B_j(z,w)\partial_{w_j}$, which are tangent to the model surface and vanish at the origin, \cite{BER, Beloshapka2004}. Applying necessary computations, one finds that this algebra  is generated by the real part of the {\it dilation} vector field (see \eqref{dilation} for its flow)
\begin{equation}\label{dilation-inf}
X=z\,\partial_z+2\,w_1\,\partial_{w_1}+3\,w_2\,\partial_{w_2}+3\,w_3\,\partial_{w_3}+4\,w_4\,\partial_{w_4}.
\end{equation}

\subsection{Branch A$^{\prime\prime}$}

Now we assume that the lifted invariant $\bb=\frac{1}{4} V^4_{Z^2\oZ^2}$ vanishes at the origin\footnote{We emphasize that, for the sake of generality, we do not assume here that $\bb$ {\it locally} vanishes in a neighborhood of the origin.}. In this case, the recurrence relation of $\im V^4_{Z^2\oZ U_2}$ in \eqref{rec-rel-ord-4-2-eq1} is not anymore of use to normalize $\alpha^1_{U_4}$ (cf. Lemma \ref{lem-Branch-A-alpha1-U4}). Nevertheless, we still have the opportunity of normalizing this Maurer-Cartan form in order five, where we have
\[
\dt V^2_{Z^4\oZ}\equiv -4\i \alpha^1_{U_4}-\frac{1}{2}\,V^2_{Z^4\oZ}\,\alpha^4_{U_4}.
\]
This recurrence equation permits us to solve its imaginary part for $\alpha^1_{U_4}$ after setting $\im V^2_{Z^4\oZ}=0$. More generally, we have

\begin{Lemma}
\label{lem-alpha1-U4-A'}
For every $l\geq 0$, one can normalize the Maurer--Cartan form $\alpha^1_{U_4^{l+1}}$ by setting $\im V^2_{Z^4\oZ U_4^l}=0.$
\end{Lemma}

As in branch {\bf A$^\prime$}, it now remains only to normalize the Maurer--Cartan form $\alpha^4_{U_4}$. Indeed, the above lemma finalizes the construction of the desired complete normal form of Branch A$^{\prime\prime}$ introduced in the main Theorem \ref{main-result}. 

Seeking for possible normalization of the single form $\alpha^4_{U_4}$, we first turn to the recurrence relations of the two remaining fourth order invariants, namely $\im V^4_{Z^2\oZ U_2}$ and $V^4_{Z^2\oZ U_1}$ given in \eqref{rec-rel-ord-4-2}. Since these two invariants are now relative, we divide this branch into the following three subbranches

\begin{description}
\item[Branch A$^{\prime\prime}$-1] $V^4_{Z^2\oZ U_1}\neq 0$.
\item[Branch A$^{\prime\prime}$-2] $V^4_{Z^2\oZ U_1} = 0$ but $\im V^4_{Z^2\oZ U_2}\neq 0$.
\item[Branch A$^{\prime\prime}$-3] $V^4_{Z^2\oZ U_1}, \im V^4_{Z^2\oZ U_2} = 0$.
\end{description}

 Let us proceed to the investigation of these subbranches.

\subsection{Branch A$^{\prime\prime}$-1}
 This branch is analogous to Branch {\bf A$^\prime$-1} as the normalization of $\alpha^4_{U_4}$ can again be achieved by setting $\|V^4_{Z^2\oZ U_1}\| =1$. Consequently, we recover the complete normal form obtained in Theorem \ref{Theorem-A-1} subject to the following two modifications: $1)$ the fourth defining equation in \eqref{NM-Branch-A-1-NF} is now given by
\[
\aligned
v^4&=\frac{\rm i}{6} \, (z^3\oz- z\oz{}^3) +\frac{\exp(\i\theta)}{2} z^2\oz u_1+\frac{\exp(-\i\theta)}{2} z\oz{}^2 u_1
\\
&\hspace{2.5cm} +\frac{\i \im V^4_{Z^2\oZ U_2}}{2} (z^2\oz u_2-z\oz^2u_2) +\sum_{j+k+|\ell |\geq 5} \frac{V^4_{Z^j\oZ{}^k U^\ell}}{j! k! \ell!} z^j\oz{}^k u^\ell,
\endaligned
\]
 and $2)$ in the corresponding cross-section \eqref{cross-sec-A-1-+}, substitute $\im V^4_{Z^2\oZ U_2U_4^l}=0$ with $\im V^2_{Z^4\oZ U_4^l}=0$.

The isotropy group at the origin associated with the manifolds of this branch are discrete and thus, the presented normal form will be unique up to these discrete groups.

\subsection{Branch A$^{\prime\prime}$-2}

We now assume that $V^4_{Z^2\oZ U_1} = 0$ but $\im V^4_{Z^2\oZ U_2} \neq 0$.  In this branch the recurrence relation \eqref{rec-rel-ord-4-2-eq1} reduces to
\[
\dt(\im V^4_{Z^2\oZ U_2}) \equiv -\frac{1}{2} \im V^4_{Z^2\oZ U_2} \alpha^4_{U_4}.
\]
Then as before, setting $\|\im V^4_{Z^2\oZ U_2}\|=1$ immediately enables us to normalize the last Maurer--Cartan form $\alpha^4_{U_4}$.

Additionally, in this branch, examining the coefficients of the independent horizontal forms in the recurrence relation for $V^4_{Z^2\oZ U_1} = 0$ imposes the following extra constraints on the order five differential invariants
\begin{align}
V^4_{Z^2\oZ U_1U_4} &= 0,\notag \\
V^4_{Z^2\oZ U_1 U_3} &= \frac{1}{12}(\i + 4V^4_{Z^2\oZ{}^2})V^3_{Z\oZ{}^3U_3} + \frac{1}{12}(\i + V^4_{Z^2\oZ{}^2} + 3\i (V^4_{Z^2\oZ{}^2})^2) V^3_{Z^3\oZ U_3},\notag \\
V^4_{Z^2\oZ U_1U_2} &= -\frac{1}{4}(V^4_{Z^2\oZ{}^2}-\i (V^4_{Z^2\oZ^2})^2)V^3_{Z^3\oZ{} U_2} + \frac{1}{4}(\i + V^4_{Z^2\oZ{}^2}) V^3_{Z^2\oZ{}^2U_2}\notag \\
V^4_{Z^2\oZ U_1^2} &= \frac{1}{4}(1+\i V^4_{Z^2\oZ{}^2})V^2_{Z^2\oZ{}^2U_1}+\frac{1}{4}(\i + V^4_{Z^2\oZ{}^2})V^3_{Z^2\oZ{}^2U_1},\label{A'-1p constraints}\\
V^4_{Z^2\oZ{}^2 U_1} &= -2\i\epsilon-(\re V^3_{Z^2\oZ^3}-\im V^3_{Z^2\oZ^3}) (V^4_{Z^2\oZ^2})^3 \notag
\\
&-\frac{1}{12}\,(11\,\re V^2_{Z^2\oZ^3}+5\,\im V^3_{Z^2\oZ^3}+16\,\im V^3_{Z\oZ^4}) (V^4_{Z^2\oZ^2})^2\notag
\\
&+\frac{1}{12}\, (10\,\im V^2_{Z^2\oZ^3}-12\, \re V^3_{Z^2\oZ^3}+4\,\re V^3_{Z\oZ^4}+96\,\epsilon) V^4_{Z^2\oZ^2} \notag
\\
&+\frac{1}{12}\,(2\,\re V^2_{Z^2\oZ^3}-2\,\im V^3_{Z^2\oZ^3}-\im V^3_{Z\oZ^4}),\notag \\
\re V^3_{Z\oZ{}^4} &= 24\,\epsilon-3\,(\im V^2_{Z^2\oZ^3}+\re V^3_{Z^2\oZ^3}) (V^4_{Z^2\oZ^2})^2 \notag
\\
&-2\, (\re V^2_{Z^2\oZ^3}+2\, \im V^3_{Z^2\oZ^3}+2\,\im V^3_{Z\oZ^4}) V^4_{Z^2\oZ^2}+2\,\im V^2_{Z^2\oZ^3}-4\,\re V^3_{Z^2\oZ^3},\notag \\
V^4_{Z^3\oZ U_1} &= \frac{1}{12} (2-\i V^4_{Z^2\oZ{}^2}) \re V^2_{Z^2\oZ{}^3} + \frac{1}{12} (\i+4 V^4_{Z^2\oZ{}^2}) V^3_{Z^2\oZ{}^3} + \frac{1}{4}(\i + V^4_{Z^2\oZ{}^2})V^3_{Z^3\oZ{}^2}\notag \\
&-\frac{1}{12}(1+\i V^4_{Z^2\oZ{}^2} + 3(V^4_{Z^2\oZ{}^2})^2)\re V^2_{Z^4\oZ}
+ \frac{1}{12}(\i + V^4_{Z^2\oZ{}^2} + 3\i (V^4_{Z^2\oZ{}^2})^2) V^3_{Z^3\oZ},\notag
\end{align}
where $\epsilon=\rm{sgn}(\im V^4_{Z^2\oZ U_2})$ depends on the sign of the relative invariant $\im V^4_{Z^2\oZ U_2}$.
Further prolongations provide expressions for $V^4_{Z^{2+j} \oZ{}^{1+k} U_1^{1+l_1} U_2^{l_2} U_3^{l_3} U_4^{l_4}}$. Hence, as in the previous subbranch, all Maurer--Cartan forms are normalized here and we have

\begin{Theorem}\label{Theorem-A'-2}
Let $M\subset\mathbb C^4$ be a $6$-dimensional totally nondegenerate manifold belonging to Branch {\bf A$^{\prime\prime}$-2}. There exists a transformation of $\mathbb{C}^4$ locally mapping $M$ to the normal form
\begin{align*}
v^1&=z\oz+\sum_{j+k+|\ell |\geq 5} \frac{V^1_{Z^j\oZ{}^k U^\ell}}{j! k! \ell!} z^j\oz{}^k u^\ell,\nonumber
\\
v^2&=\frac{1}{2}\,(z^2\oz+z\oz{}^2)+\sum_{j+k+|\ell |\geq 5} \frac{V^2_{Z^j\oZ{}^k U^\ell}}{j! k! \ell!} z^j\oz{}^k u^\ell,
\\
v^3&=-\frac{\rm i}{2}\,(z^2\oz-z\oz^2)+\sum_{j+k+|\ell |\geq 5} \frac{V^3_{Z^j\oZ{}^k U^\ell}}{j! k! \ell!} z^j\oz{}^k u^\ell,\nonumber
\\
v^4&=\frac{\rm i}{6} \, (z^3\oz- z\oz{}^3)+\frac{\i\epsilon}{2}\, \big(z^2\oz u_2-z\oz{}^2 u_2\big)
+\sum_{j+k+|\ell |\geq 5} \frac{V^4_{Z^j\oZ{}^k U^\ell}}{j! k! \ell!} z^j\oz{}^k u^\ell,\nonumber
\end{align*}
where the coefficients $V_J$ of order $|J| \geq 5$ satisfy the partial cross-section $\mathcal K$ in \eqref{partial-cross-sec}-\eqref{cross-sec-A-1-+}, subject to the substitution $\im V^4_{Z^2\oZ U_2U_4^l}=0$ with $\im V^2_{Z^4\oZ U_4^l}=0$, along with the order five relations\footnote{These relations must be evaluated subject to the branch hypothesis $V^4_{Z^2\oZ^2}(0)=0$, bearing in mind that our defining normal forms are power series expansions around the origin.} \eqref{A'-1p constraints}.
 In this branch, $M$ admits trivial isotropy group at the origin.
\end{Theorem}

\subsection{Branch A$^{\prime\prime}$-3}

We now assume that both $V^4_{Z^2\oZ U_1}$ and $\im V^4_{Z^2\oZ U_2}$ vanish, identically. Consequently, no further normalization is available at order four. The effect of vanishing $V^4_{Z^2\oZ U_1}$ on the order five invariants are given in \eqref{A'-1p constraints} with $\epsilon =0$. Incorporating these expressions, the coefficients of the horizontal forms in the recurrence relation for $\im V^4_{Z^2\oZ U_2} = 0$ yield in addition
\begin{align}
&V^4_{Z^2\oZ U_2U_4} = 0,\notag \\
&V^4_{Z^2\oZ U_2U_3} = \frac{\i}{64}\big(\re V^3_{Z^3\oZ U_3}\,\im V^4_{Z\oZ^4}-\im V^2_{Z^4\oZ U_3}\big) (V^4_{Z^2\oZ^2})^2,\label{A-2 constraints} \\
&V^4_{Z^2\oZ U_2U_2} = -\frac{1}{128}\big(\i\re V^4_{Z^3\oZ U_2} (\im V^3_{Z\oZ^4}-3)-4\i\im V^2_{Z^4\oZ U_2}\big) (V^4_{Z^2\oZ^2})^2,\notag \\
&V^3_{Z^2\oZ^2 U_2} = {\sf P}(V_J),\notag \\
&V^4_{Z^3\oZ U_2} = {\sf Q}(V_J),\notag
\end{align}
where $\sf P$ and $\sf Q$ are two respectively real and complex polynomials in terms of the lifted invariants $V_J$ of order $|J|=5$.
At order five, we are thus left with the set of independent {\it relative} invariants
\begin{multline*}
\aligned
\mathcal{R}_\textbf{A$^{\prime\prime}$-3} = \{V^1_{Z^3\oZ{}^2}, &V^1_{Z^2\oZ{}^2U_1}, \re V^2_{Z^4\oZ}, \re V^2_{Z^3\oZ{}^2}, V^2_{Z^2\oZ{}^2U_1}, \im V^3_{Z^4\oZ}, V^3_{Z^3\oZ{}^2}, \im V^3_{Z^3\oZ U_2}, \\
& \ \ \ V^3_{Z^3\oZ U_3}, V^3_{Z^2\oZ{}^2U_1}, V^3_{Z^2\oZ U_1^2}, V^4_{Z^4\oZ}, V^4_{Z^3\oZ{}^2}, V^4_{Z^2\oZ{}^2U_2}, V^4_{Z^2\oZ{}^2 U_3}, V^4_{Z^2\oZ{}^2 U_4}\},
\endaligned
\end{multline*}
where their recurrence relations are exactly as those displayed in \eqref{Rec-rel-ord-5-A-2}. Thus, similar to Branch {\bf A$^\prime$-2}, we observe that each recurrence relation $\dt R$, with $R\in \mathcal{R}_\textbf{A$''$-3}$ allows for the normalization of $\alpha^4_{U_4}$ if $R\neq 0$.  Accordingly, we divide the normalization process into the two subbranches
\begin{description}
  \item[{\bf Branch A$^{\prime\prime}$-3-1}] If at least one relative invariant in $\mathcal{R}_\textbf{A$^{\prime\prime}$-3}$ is nonzero.
  \item[{\bf Branch A$^{\prime\prime}$-3-2}] If all relative invariants in $\mathcal{R}_\textbf{A$^{\prime\prime}$-3}$ locally vanish.
\end{description}

Let us continue along these two new subbranches.

\subsubsection{{\bf Branch A$^{\prime\prime}$-3-1}}

By the assumption, there exists a {\it nonzero} relative invariant $R \in \mathcal{R}_\textbf{A$^{\prime\prime}$-3}$. Then, we can set $\| R \| = 1$ and solve the recurrence equation of $d (\| R \|)$ in \eqref{Rec-rel-ord-5-A-2} to normalize $\alpha^4_{U_4}$.

\begin{Theorem}
\label{Theorem-A-2-1}
Let $M\subset\mathbb C^4$ be a $6$-dimensional totally nondegenerate manifold belonging to Branch {\bf A$^{\prime\prime}$-3-1} with some nonzero relative invariant $R\in \mathcal{R}_\textbf{A$^{\prime\prime}$-3}$. There exists a certain transformation of $\mathbb{C}^5$, mapping $M$ locally to the normal form
\begin{equation*}
\aligned
v^1&=z\oz+\sum_{j+k+|\ell |\geq 5} \frac{V^1_{Z^j\oZ{}^k U^\ell}}{j! k! \ell!} z^j\oz^k u^\ell,
\\
v^2&=\frac{1}{2}\,(z^2\oz+z\oz^2)+\sum_{j+k+|\ell |\geq 5} \frac{V^2_{Z^j\oZ{}^k U^\ell}}{j! k! \ell!} z^j\oz^k u^\ell,
\\
v^3&=-\frac{\rm i}{2}\,(z^2\oz-z\oz^2)+\sum_{j+k+|\ell |\geq 5} \frac{V^3_{Z^j\oZ{}^k U^\ell}}{j! k! \ell!} z^j\oz^k u^\ell,
\\
v^4&=\frac{\rm i}{6} \, (z^3\oz- z\oz^3)+\sum_{j+k+|\ell |\geq 5} \frac{V^4_{Z^j\oZ{}^k U^\ell}}{j! k! \ell!} z^j\oz^k u^\ell,
\endaligned
\end{equation*}
where the coefficients $V_J$ and their conjugations satisfy the cross-section normalizations \eqref{partial-cross-sec} and \eqref{cross-sec-A-1-+}, subject to replacing $\im V^4_{Z^2\oZ U_2U_4^l}=0$ by $\im V^2_{Z^4\oZ U_4^l}=0$ in \eqref{cross-sec-A-1-+}, together with $\| R \| = 1$. The coefficients of this normal form enjoy the relations \eqref{A'-1p constraints}-\eqref{A-2 constraints} and their prolongations with $\epsilon=0$. It is also unique up to a discrete group of transformations.
\end{Theorem}

\subsubsection{\bf Branch A$^{\prime\prime}$-3-2}

We now assume that all relative invariants in $\mathcal{R}_\textbf{A$^{\prime\prime}$-3}$ vanish, identically.
In this branch, one finds the recurrence relation of $\bb$ simply as $\dt \bb=0$. Hence, it is nothing but a constant real function which vanishes at the origin, i.e. the {\it zero function}
\[
\bb\equiv 0.
\]
 As in Branch \textbf{A$^\prime$-2-2}, checking the coefficients of the horizontal forms in the recurrence relations of the vanishing relative invariants in $\mathcal{R}_\textbf{A$^{\prime\prime}$-3}$ reveals that only a few lifted differential invariants remain nonzero in order six. Among them, we consider the {\it relative} invariant $\im V^3_{Z^3\oZ U_1U_3}$ with the recurrence relations --- note that at this stage, the real counterpart $\re V^3_{Z^3\oZ U_1U_3}$ of this invariant vanishes as well as many other invariants
\begin{equation}\label{Rec-rel-ord-6-A-2-2}
\begin{aligned}
\dt (\im V^3_{Z^3\oZ U_1U_3})\equiv \frac{3}{2}\, \im V^3_{Z^3\oZ U_1U_3}\,\alpha^4_{U_4}.
\end{aligned}
\end{equation}
 Accordingly, we shall divide once again two new subbranches
\begin{description}
  \item[{\bf Branch A$^{\prime\prime}$-3-2-1}] If the real invariant $\im V^3_{Z^3\oZ U_1U_3}$ is nonzero at the origin.
  \item[{\bf Branch A$^{\prime\prime}$-3-2-2}] If $\im V^3_{Z^3\oZ U_1U_3}$ locally vanishes around the origin.
\end{description}

In light of the recurrence relation \eqref{Rec-rel-ord-6-A-2-2}, it is clear in the former branch {\bf A$^{\prime\prime}$-3-2-1} that we can normalize the last Maurer--Cartan form $\alpha^4_{U_4}$ by setting $\| \im V^3_{Z^3\oZ U_1U_3} \|=1$. Indeed, we have

\begin{Theorem}
\label{Theorem-A-2-2-1}
Let $M\subset\mathbb C^4$ be a $6$-dimensional totally nondegenerate manifold belonging to Branch {\bf A$^{\prime\prime}$-3-2-1}. Then, there exists some origin-preserving transformation, mapping $M$ to the  normal form
\begin{equation}\label{NM-Branch-A-2-2-NF}
\aligned
v^1&=z\oz+\sum_{j+k+|\ell |\geq 6} \frac{V^1_{Z^j\oZ{}^k U^\ell}}{j! k! \ell!} z^j\oz^k u^\ell,
\\
v^2&=\frac{1}{2}\,(z^2\oz+z\oz^2)+\sum_{j+k+|\ell |\geq 6} \frac{V^2_{Z^j\oZ{}^k U^\ell}}{j! k! \ell!} z^j\oz^k u^\ell,
\\
v^3&=-\frac{\rm i}{2}\,(z^2\oz-z\oz^2)+\sum_{j+k+|\ell |\geq 6} \frac{V^3_{Z^j\oZ{}^k U^\ell}}{j! k! \ell!} z^j\oz^k u^\ell,
\\
v^4&=\frac{\rm i}{6} \, (z^3\oz- z\oz^3)+\sum_{j+k+|\ell |\geq 6} \frac{V^4_{Z^j\oZ{}^k U^\ell}}{j! k! \ell!} z^j\oz^k u^\ell,
\endaligned
\end{equation}
where the coefficients $V_J$ and their conjugations enjoy the cross-section \eqref{partial-cross-sec}-\eqref{cross-sec-A-1-+}, replacing $\im V^4_{Z^2\oZ U_2U_4^l}=0$ by $\im V^2_{Z^4\oZ U_4^l}=0$ in \eqref{cross-sec-A-1-+}, along with $\| \im V^3_{Z^3\oZ U_1U_3} \| = 1$. The isotropy group of $M$ at the origin is discrete and the normal form is unique up to it.
\end{Theorem}

But, in Branch {\bf A$^{\prime\prime}$-3-2-2}, where the order six lifted differential invariant $\im V^3_{Z^3\oZ U_1U_3}$ locally vanishes, the recurrence relation \eqref{Rec-rel-ord-6-A-2-2} does not help anymore to normalize the last Maurer--Cartan form $\alpha^4_{U_4}$.
By a similar argument as that we presented in Branch {\bf A$^\prime$-2-2-2}, we prove that the normalization of $\alpha^4_{U_4}$ is impossible even in higher orders.

Indeed, after performing necessary computations, we find in this branch the same constant type structure equations \eqref{struc-eq-A-2-2-2} with $\bb = 0$.
Thus, we shall prolong the equivalence problem to our $6$-dimensional surfaces $M$ of the current branch to the $7$-dimensional prolonged spaces $M\times\mathcal G^{\sf red}$, where $\mathcal G^{\sf red}$ is the $1$-dimensional reduced subgroup of $\mathcal G$ identified by the Maurer--Cartan form $\alpha^4_{U_4}$. Our computations show in addition that
\begin{equation}\label{dalpha4-U4-A'}
\dt \alpha^4_{U_4}=0.
\end{equation}
Thus in this branch, the equivalence problem is encoded by the {\it constant type} structure equations \eqref{struc-eq-A-2-2-2}-\eqref{dalpha4-U4-A'}, with $\bb = 0$, on the prolonged space $M\times\mathcal G^{\sf red}$. Again by Cartan's classical results, all surfaces of this branch are equivalent. Among them, we have the {\it single} Beloshapka's model $M(\frac{\rm i}{6}, 0)$ with the defining equations (cf. \eqref{def-eq-model})
\begin{equation}\label{model-A'-2-2-2}
\aligned
v^1&=z\oz,
\\
v^2&=\frac{1}{2}\,(z^2\oz+z\oz^2),
\\
v^3&=-\frac{\rm i}{2}\,(z^2\oz-z\oz^2),
\\
v^4&=\frac{\rm i}{6}\, (z^3\oz- z\oz^3).
\endaligned
\end{equation}

\begin{Theorem}
\label{Th-A-2-2-2}
Every $6$-dimensional totally nondegenerate manifold $M$ belonging to Branch {\bf A$^{\prime\prime}$-3-2-2} is biholomorphically equivalent to the model surface $M(\frac{\rm i}{6}, 0)$ defined in \eqref{model-A'-2-2-2}. The isotropy group of such surfaces is $1$-dimensional with the infinitesimal generator \eqref{dilation-inf}. The transformations mapping $M$ to $M(\frac{\rm i}{6}, 0)$ are unique up to the action of this isotropy group.
\end{Theorem}

\section{Branch B: $V^4_{Z^3\oZ} = 0$ and $V^4_{Z^2\oZ{}^2} \neq 0$}
\label{sec-Branch-B}

We now turn to Branch {\bf B}, where the fourth order differential invariant $V^4_{Z^3\oZ}$ vanishes locally while, by the total nondegeneracy, $V^4_{Z^2\oZ{}^2}$ remains nonzero. Recall that at this stage, all the lifted differential invariants appearing in \eqref{partial-cross-sec} have been normalized and the Maurer--Cartan forms that remain to be normalized are listed in \eqref{MC-3}.

Let us examine first the recurrence relation of the vanished invariant $V^4_{Z^3\oZ}$:
\[
\aligned
0 &=\dt V^4_{Z^3\oZ}= \big(V^4_{Z^4\oZ}+\frac{\rm i}{2}\,V^3_{Z^3\oZ}V^4_{Z^2\oZ^2}\big)\,\omega^Z+\big(V^4_{Z^3\oZ{}^2}-\frac{\rm i}{2}\,V^3_{Z^3\oZ}V^4_{Z^2\oZ^2}\big)\,\omega^{\oZ}
\\
&+\big(V^4_{Z^3\oZ U_1}+V^3_{Z^3\oZ} \,\im V^4_{Z^2\oZ U_1}\big)\,\omega^{1}+ \big(V^4_{Z^3\oZ U_2}+ V^3_{Z^3\oZ} \,\im V^4_{Z^2\oZ U_2}\big)\,\omega^{2}+V^4_{Z^3\oZ U_3}\,\omega^{3}+V^4_{Z^3\oZ U_4}\,\omega^{4}.
\endaligned
\]
As a result, equating to zero the coefficients of linearly independent horizontal forms gives in turns the following constraints on the order five differential invariants
\begin{equation}\label{relation-ord-5}
\aligned
 & V^4_{Z^4\oZ}=-\frac{\rm i}{2}\,V^3_{Z^3\oZ}V^4_{Z^2\oZ^2}, \qquad V^4_{Z^3\oZ{}^2}=\frac{\rm i}{2}\,V^3_{Z^3\oZ}V^4_{Z^2\oZ^2},
  \\
 &  V^4_{Z^3\oZ U_1}=-V^3_{Z^3\oZ}\,\im V^4_{Z^2\oZ U_1}, \qquad   V^4_{Z^3\oZ U_2} = - V^3_{Z^3\oZ} \,\im V^4_{Z^2\oZ U_2}, \qquad V^4_{Z^3\oZ U_3}=V^4_{Z^3\oZ U_4} = 0.
\endaligned
\end{equation}

We now continue the normalization process by inspecting the order four recurrence relations \eqref{rec-ord-4}. As the first consequence, one may normalize the Maurer--Cartan form $\alpha^3_{U_3}$ by solving the recurrence relation
\[
\dt V^4_{Z^2\oZ{}^2}\equiv V^4_{Z^2\oZ{}^2}\big(\alpha^4_{U_4}-\frac{4}{3} \, \alpha^3_{U_3}\big)
\]
after setting $V^4_{Z^2\oZ^2}=1$, as is permitted in this branch. More generally we have

\begin{Lemma}
\label{lem-Branch-B-alpha3-U-3}
In Branch {\bf B} and for each $k\geq 0$, one can normalize the Maurer--Cartan form $\alpha^3_{U_3^{k+1}}$ by setting $V^4_{Z^2\oZ{}^2 U_3^{k}}=\delta^k_0$.
\end{Lemma}

Next, by setting $V^4_{Z^2\oZ U1}=0$, its corresponding recurrence relation in \eqref{rec-ord-4} can be readily solved to normalize the two real Maurer--Cartan forms $\alpha^2_{U_4}$ and $\alpha^3_{U_4}$. Likewise, the recurrence relation of $\im V^4_{Z^2\oZ U_2}$ provides the normalization of $\alpha^1_{U_4}$ upon setting this invariant to zero. More generally we have

\begin{Lemma}
\label{lem-Branch-B-alpha-1,2,3-U-4}
For each $k, l\geq 0$, one can normalize
\begin{itemize}
  \item[1)] two real Maurer--Cartan forms $\alpha^2_{U_3^k U_4^{l+1}}$ and $\alpha^3_{U_3^k U_4^{l+1}}$ by setting $V^4_{Z^2\oZ U_1 U_3^k U_4^l}=0$.
  \item[2)] the real Maurer--Cartan form $\alpha^1_{U_4^{l+1}}$ by setting $\im V^4_{Z^2\oZ U_2U_4^l}=0$.
\end{itemize}
\end{Lemma}

At this stage of the process, the list of the remaining unnormalized Maurer-Cartan forms \eqref{MC-3} is reduced to
\begin{equation}\label{MC-B-before-branch}
\im \mu_{U_3^{k+1} U_4^l}, \qquad \mu_{U_4^l}, \qquad \omu_{U_4^l}, \qquad \alpha^2_{U_3^{k+1}}, \qquad \alpha^4_{U_4^l}.
\end{equation}

Among the recurrence relations \eqref{rec-ord-4}, then it remains just the consideration of the first relation which now is simplified to
\begin{equation}\label{rec-ord-4-remained-Branch-B}
\aligned
\\
\dt V^3_{Z^3\oZ}\equiv&\; V^3_{Z^3\oZ}\,\big(3\i \alpha^2_{U_3} - \frac{1}{4}\,\alpha^4_{U_4}\big).
\endaligned
\end{equation}
Then at this order, the normalization of the two real Maurer--Cartan forms $\alpha^2_{U_3}$ and $\alpha^4_{U_4}$ depends on weather the {\it relative} invariant $V^3_{Z^3\oZ}$ vanishes or not. Thus, we have to divide the process into the following two subbranches
\begin{description}
  \item[{\bf Branch B-1.}] If $V^3_{Z^3\oZ}\neq 0$,
  \item[{\bf Branch B-2.}] If $V^3_{Z^3\oZ} = 0$.
\end{description}

As in Branch {\bf A} and before proceeding along the above two subbranches, let us consider the following order five recurrence relations

\[
\aligned
\dt V^3_{Z^2\oZ{}^2U_1}&\equiv {\rm Im}\,\mu_{U_4}+\cdots,
\\
\dt V^4_{Z^2\oZ{}^2U_1}&\equiv -\frac{8}{3}\,\im\mu_{U_3}+\cdots,
\\
\dt V^4_{Z^2\oZ{}^2U_2}&\equiv {\rm Re}\,\mu_{U_4}+\cdots,
\endaligned
\]
where the $"\cdots"$ parts consist of the other Maurer--Cartan forms. By setting $V^3_{Z^2\oZ{}^2U_1}=V^4_{Z^2\oZ{}^2U_1}=V^4_{Z^2\oZ{}^2U_2} = 0$, these equations enable one to solve them for the Maurer-Cartan forms $\mu_{U_4}$ and $\im \mu_{U_3}$.

Following these normalizations, one also receives the recurrence relations
\[
\aligned
dV^1_{Z^2\oZ^2U_2}&\equiv -32\,\im \mu_{U_3^2}+\cdots,
\\
dV^2_{Z^2\oZ^2 U_1}&\equiv -\frac{4}{3}\,\alpha^2_{U_3^2}+\cdots,
\\
dV^4_{Z^2\oZ^2U_4}&\equiv \alpha^4_{U_4^2}+\cdots,
\endaligned
\]
which provide the normalization of the Maurer-Cartan forms $\re \mu_{U_2^2}, \alpha^2_{U_3^2}$ and $\alpha^4_{U_4^2}$ by setting $V^1_{Z^2\oZ^2U_3}=V^2_{Z^2\oZ^2 U_1}=V^4_{Z^2\oZ^2U_4}=0$. Summing up, we have more generally

\begin{Lemma}
For each $k, l \geq 0$, one can normalize the Maurer--Cartan forms
\renewcommand\labelenumi{\theenumi)}
\begin{enumerate}
 \item $\re \mu_{U_4^{l+1}}$ and $\im \mu_{U_4^{l+1}}$ by setting $V^4_{Z^2\oZ{}^2U_2 U_4^l}=0$ and $V^3_{Z^2\oZ{}^2U_1 U_4^l}=0$, respectively,
\item  $\im \mu_{U_3 U_4^l}$ by setting $V^4_{Z^2\oZ{}^2U_1 U_4^l}=0$,
\item   $\im \mu_{U_3^{j+2} U_4^l}$ by setting $V^1_{Z^2\oZ{}^2U_2 U_3^k U_4^l}=0$,
\item $\alpha^2_{U_3^{j+2}}$ by setting $V^2_{Z^2\oZ{}^2U_1 U_3^k}=0$,
\item $\alpha^4_{U_4^{l+2}}$ by setting $V^4_{Z^2\oZ{}^2U_4^{l+1}}=0$.
\end{enumerate}
\end{Lemma}

Thank to this lemma, the list \eqref{MC-B-before-branch} of the not yet normalized Maurer--Cartan forms is now reduced to only two
\begin{equation}\label{MC-B-B}
\alpha^2_{U_3} \qquad {\rm and} \qquad \alpha^4_{U_4}.
\end{equation}
We emphasize at this stage that we have finished construction of the desired complete normal form of Branch {\bf B} which was introduced in Theorem \ref{main-result}.
Also remark that after applying the above normalizations, the relations \eqref{relation-ord-5} are now simplified to
\begin{equation}\label{relation-ord-5-simplified}
\aligned
  V^4_{Z^4\oZ}=-\frac{\rm i}{2}\,V^3_{Z^3\oZ}, \qquad V^4_{Z^3\oZ{}^2}=\frac{\rm i}{2}\,V^3_{Z^3\oZ}, \qquad  V^4_{Z^3\oZ U_j} &= 0 \qquad j=1,\ldots,4.
\endaligned
\end{equation}
These relations will be helpful in simplifying the subsequent computations.

\subsection{Branch B-1}

Under the assumption $V^3_{Z^3\oZ} \neq 0$, we set this relative invariant to $\i$ and apply the recurrence relation \eqref{rec-ord-4-remained-Branch-B} to normalize the two remaining real Maurer--Cartan forms $\alpha^4_{U_4}$ and $\alpha^2_{U_3}$. This finishes the process of normalizations in this branch, yielding a complete moving frame. The normal form, corresponding to this moving frame is unique modulo the action of a discrete isotropy group at the origin. To summarize, we have

\begin{Theorem}
  Every $6$-dimensional totally nondegenerate submanifold of $\mathbb C^4$, belonging to Branch {\bf B-1}, can be mapped through some origin-preserving transformation, which is unique up to the action of some discrete isotropy group at the origin, to the normal form
  \begin{equation*}
\aligned
v^1&=z\oz+\sum_{j+k+|\ell|\geq 5} \frac{V^1_{Z^j\oZ{}^k U^\ell}}{j! k! \ell!} z^j\oz^k u^\ell,
\\
v^2&=\frac{1}{2}\,(z^2\oz+z\oz^2)+\sum_{j+k+|\ell|\geq 5} \frac{V^2_{Z^j\oZ{}^k U^\ell}}{j! k! \ell!} z^j\oz^k u^\ell,
\\
v^3&=-\frac{{\rm i}}{2}\,(z^2\oz-z\oz^2)+\frac{\i}{6} \big(z^3\oz- z\oz^3\big)
 +\sum_{j+k+|\ell|\geq 5} \frac{V^3_{Z^j\oZ{}^k U^\ell}}{j! k! \ell!} z^j\oz^k u^\ell,
\\
v^4&=\frac{1}{4} z^2\oz^2+\sum_{j+k+|\ell|\geq 5} \frac{V^4_{Z^j\oZ{}^k U^\ell}}{j! k! \ell!} z^j\oz^k u^\ell,
\endaligned
\end{equation*}
enjoying the relations \eqref{relation-ord-5-simplified} and their prolongations with $V^3_{Z^3\oZ}=\i$. Regarding the conjugation relation, the coefficients $V_J$ of this normal form enjoy the cross-section $\mathcal K$ in \eqref{partial-cross-sec} together with
\begin{equation}\label{normalizations-B-extra}
\aligned
0 &= V^4_{Z^2\oZ{}^2U_3^{k+1}}=V^4_{Z^2\oZ U_1U_3^kU_4^l}=\im V^4_{Z^2\oZ U_2U_4^l}=V^4_{Z^2\oZ{}^2 U_1U_4^l}=V^4_{Z^2\oZ{}^2 U_2U_4^l}=V^3_{Z^2\oZ{}^2 U_1U_4^l}
\\
&=V^1_{Z^2\oZ^2 U_2 U_3^k U_4^l}=V^2_{Z^2\oZ^2 U_1 U_3^k}=V^4_{Z^2\oZ^2U_4^{l+1}},
\endaligned
\end{equation}
for $k,l\geq 0$. In addition, the isotropy group of $M$ at the origin is zero dimensional.
\end{Theorem}

\subsection{Branch B-2}

The assumption $V^3_{Z^3\oZ} = 0$ of this branch prevents us from normalizing the remaining Maurer--Cartan forms \eqref{MC-B-B} in order four. Nevertheless, it brings some computationally facilitating outputs. Indeed, the left hand sides of the six equations in \eqref{relation-ord-5-simplified} vanish identically now. Consequently, the homogeneous system formed by the coefficients of the horizontal forms in the recurrence relations of these equations, combined with those arising from $dV^3_{Z^3\oZ}=0$, surprisingly reveals the solution
\begin{equation}\label{relation-order-s-3}
\aligned
&\re V^2_{Z^3\oZ{}^2}=-\im V^3_{Z^3\oZ{}^2}=\frac{1}{2}\,V^4_{Z^3\oZ^3},
\\
& V^2_{Z^4\oZ}=\i V^3_{Z^4\oZ}=V^4_{Z^4\oZ^2}
\\
&\re V^3_{Z^3\oZ^2}=\im V^2_{Z^3\oZ^2}=V^3_{Z^3\oZ U_j}=V^4_{Z^5\oZ}=V^4_{Z^3\oZ^2U_j}=V^4_{Z^3\oZ U_jU_k}=V^4_{Z^4\oZ U_j}= 0,
\endaligned
\end{equation}
for $j, k=1,\ldots,4$. These equations significantly simplify upcoming computations at orders five and six. Additionally, we observe that applying the recurrence formula on the first two lines of \eqref{relation-order-s-3} gives in turn some further relations in order six --- for saving the space, we do not present them here.

Now, taking into account the normalizations applied in the former orders along with the relations \eqref{relation-ord-5-simplified} and \eqref{relation-order-s-3}, it remains in order five to consider the following ten recurrence relations
\begin{align}
\dt V^1_{Z^3\oZ{}^2}&\equiv -\frac{1}{4}\,V^1_{Z^3\oZ{}^2}\,\big(3\,\alpha^4_{U_4}-4{\rm i}\,\alpha^2_{U_3}\big),\nonumber
\\
\dt V^1_{Z^2\oZ{}^2 U_1}&\equiv -V^1_{Z^2\oZ{}^2U_1}\,\alpha^4_{U_4}+\frac{1}{3}\,\big(24\,\im V^3_{Z^2\oZ{}U_1^2}-10\,V^3_{Z^2\oZ^2 U_2}+88\,\re V^4_{Z^2\oZ{}U_1U_2}\big)\,\alpha^2_{U_3},\nonumber
\\
\dt (\im V^3_{Z^3\oZ{}^2})&\equiv -\frac{1}{2}\,\im V^3_{Z^3\oZ{}^2}\,\alpha^4_{U_4}, \label{Rec-rel-ord-5-B3}
\\
\dt V^3_{Z^2\oZ{}^2 U_2}&\equiv -V^3_{Z^2\oZ{}^2 U_2}\,\alpha^4_{U_4}+8\,\im V^4_{Z^2\oZ U_1U_2}\,\alpha^2_{U_3},\nonumber
\\
\dt V^3_{Z^2\oZ{} U_1^2}&\equiv -V^3_{Z^2\oZ U_1^2}\,\alpha^4_{U_4}+\big(2\i V^3_{Z^2\oZ U_1^2}-4\i \im V^4_{Z^2\oZ U_1U_2}-V^3_{Z^2\oZ^2 U_2}\big)\,\alpha^2_{U_3}, \nonumber
\\
\dt V^3_{Z^4\oZ}&\equiv -\frac{1}{2}\,V^3_{Z^4\oZ}\,\big(\alpha^4_{U_4}-4\i \alpha^2_{U_3}\big), \nonumber
\\
\dt V^4_{Z^2\oZ U_1^2}&\equiv -\frac{1}{4}\,V^4_{Z^2\oZ U_1^2}\,\big(3\,\alpha^4_{U_4}-4\i \alpha^2_{U_3}\big),\nonumber
\\
\dt V^4_{Z^2\oZ U_1U_2}&\equiv -V^4_{Z^2\oZ U_1U_2}\,\alpha^4_{U_4}+\frac{\rm i}{2}\,\big(4\,\re V^4_{Z^2\oZ U_1U_2}-V^3_{Z^2\oZ{}^2 U_2}\big)\, \alpha^2_{U_3},\nonumber
\\
\dt (\im V^4_{Z^2\oZ U_2U_2})&\equiv -\frac{5}{4}\,\im V^4_{Z^2\oZ U_2U_2}\,\alpha^4_{U_4}+\im V^4_{Z^2\oZ U_2U_3}\, \alpha^2_{U_3}.\nonumber
\\
\dt (\im V^4_{Z^2\oZ U_2U_3})&\equiv -\frac{5}{4}\,\im V^4_{Z^2\oZ U_2U_3}\,\alpha^4_{U_4}-\im V^4_{Z^2\oZ U_2U_2}\, \alpha^2_{U_3}.\nonumber
\end{align}

One finds straightforwardly from these relations several lifted invariants which are of relative type. In addition to them and upon close inspection of the last two equations, it turns out that for $\Gamma:=(\im V^4_{Z^2\oZ U_2U_2})^2+(\im V^4_{Z^2\oZ U_2U_3})^2$ we have
\[
d\Gamma\equiv-\frac{5}{2}\,\Gamma\,\alpha^4_{U_4}.
\]
Hence, $\Gamma$ is a relative invariant, as well. Therefore, the real lifted invariants $\im V^4_{Z^2\oZ U_2U_2}$ and $\im V^4_{Z^2\oZ U_2U_3}$ are identically either zero or nonzero.

Based on a careful inspection of the above recurrence relations \eqref{Rec-rel-ord-5-B3}, we organize the subsequent computations through the following three subbranches
\begin{description}
  \item[{\bf Branch B-2-1.}] If one of the following statements holds:
 \begin{description}
   \item[a)] Either of the five {\it relative} invariants $V^1_{Z^3\oZ{}^2}$, $V^3_{Z^4\oZ}$, $V^4_{Z^2\oZ U_1^2}$, $\re V^4_{Z^2\oZ U_1U_2}$ and $\Gamma$ does not vanish at the origin.
   \item[b)] The invariants in {\bf (a)} vanish locally but $\Lambda:=(V^3_{Z^2\oZ{}^2 U_2})^2+16(\im V^4_{Z^2\oZ U_1U_2})^2$ is nonzero at the origin.
   \item[c)] The lifted invariants in {\bf (a)} and {\bf (b)}  vanish locally but $V^3_{Z^2\oZ U_1^2}$ --- which is now a relative invariant --- is nonzero at the origin.
 \end{description}
  \item[{\bf Branch B-2-2.}] When the appearing invariants throughout Branch {\bf B-2-1} vanish locally but at least one of the two real invariants $V^1_{Z^2\oZ{}^2 U_1}$, $\im V^3_{Z^3\oZ{}^2}$ --- which are now relative --- is nonzero at the origin.
  \item[{\bf Branch B-2-3.}] When all the nine invariants appearing in \eqref{Rec-rel-ord-5-B3} vanish, identically.
\end{description}

Before proceeding along these branches, let us emphasize in the subbranch {\bf B-2-1}{\bf (b)} --- where the real invariant $\re V^4_{Z^2\oZ U_1U_2}$ vanishes --- that the recurrence relations of $V^3_{Z^2\oZ{}^2 U_2}$ and $\im V^4_{Z^2\oZ U_1U_2}$ take the form
\[
\aligned
\dt V^3_{Z^2\oZ{}^2 U_2}&\equiv -V^3_{Z^2\oZ{}^2 U_2}\,\alpha^4_{U_4}+8\,\im V^4_{Z^2\oZ U_1U_2}\,\alpha^2_{U_3},
\\
\dt(\im V^4_{Z^2\oZ U_1U_2})&\equiv -\frac{1}{2}\,V^3_{Z^2\oZ{}^2 U_2}\, \alpha^2_{U_3}-\im V^4_{Z^2\oZ U_1U_2}\,\alpha^4_{U_4}.
\endaligned
\]
Thus, by applying necessary computations, one finds that
\[
\dt \Lambda=-2\,\Lambda\,\alpha^4_{U_4},
\]
verifying that $\Lambda$ is a relative invariant. Then, by the expression of $\Lambda$, both  the real invariants $V^3_{Z^2\oZ{}^2 U_2}$ and $\im V^4_{Z^2\oZ U_1U_2}$ are identically nonzero in the subbranch {\bf B-2-1}{\bf (b)}.

Among the above subbranches, it is expected that {\bf B-2-2} will potentially originate several further subbranches. Then, to keep this paper from becoming lengthy, let us defer this branch to another investigation and consider here the first and third branches. However, notice that in this branch which one of the invariants $V_\circ=V^1_{Z^2\oZ{}^2 U_1}$, $V^2_{Z^3\oZ{}^2}$ is nonzero, one can normalize $\alpha^4_{U_4}$ by setting $\| V_\circ \|=1$. Hence, the isotropy group of each surface in Branch {\bf B-2-2} is either trivial or of dimension one.

\subsubsection{\bf Branch B-2-1}

In all subbranches {\bf (a-c)} of this branch, it is possible to normalize the remaining Maurer--Cartan forms $\alpha^2_{U_3}$ and $\alpha^4_{U_4}$. The following lemma describes how to effect these normalizations by appropriate specification of differential invariants.

\begin{Lemma}
\label{lem-branch-B-3-1}
Throughout Branch {\bf B-2-1}, one can normalize the two remaining Maurer-Cartan forms $\alpha^2_{U_3}$ and $\alpha^4_{U_4}$ via one of the following ways
 \renewcommand\labelenumi{\theenumi)}
\begin{enumerate}
  \item in subbranches {\bf (a)} and {\bf (c)}, when any of the relative invariants $V_\diamond=V^1_{Z^3\oZ{}^2}, V^3_{Z^4\oZ}, V^4_{Z^2\oZ U_1^2}, V^3_{Z^2\oZ U_1^2}$ is nonzero, then by solving its corresponding recurrence relation after setting $V_\diamond=\i$.
  \item in subbrach {\bf (a)}, if $\re V^4_{Z^2\oZ U_1 U_2}$ is nonzero at the origin, then by setting $V^4_{Z^2\oZ U_1U_2}=\varepsilon+0\i$ and solving its recurrence relation, where $\varepsilon$ is a nonzero real number with $V^3_{Z^2\oZ^2U_2}({\bf 0})\neq 4\,\varepsilon$.
  \item  in subbranches {\bf (a)} and {\bf (b)}, if either $\Gamma$ or $\Lambda$ is nonzero, then by setting respectively $(\im V^4_{Z^2\oZ U_2U_2}, \im V^4_{Z^2\oZ U_2U_3})=(1,1)$ or $(\im V^3_{Z^2\oZ^2 U_2}, \im V^4_{Z^2\oZ U_1U_2})=(1,1)$ and solving the corresponding recurrence relations.
\end{enumerate}
\end{Lemma}

Then in this branch, we succeed in constructing a complete moving frame which corresponds to a unique normal form modulo some discrete isotropy group.

\begin{Theorem}
  Let $M\subset\mathbb C^4$ be a $6$-dimensional totally nondegenerate manifold belonging to Branch {\bf B-2-1}. Then $M$ can be mapped through some origin-preserving transformation, which is unique up to the action of some discrete isotropy group at the origin, to the normal form
  \begin{equation*}
\aligned
v^1&=z\oz+\sum_{j+k+|\ell|\geq 5} \frac{V^1_{Z^j\oZ{}^k U^\ell}}{j! k! \ell!} z^j\oz^k u^\ell,
\\
v^2&=\frac{1}{2}\,(z^2\oz+z\oz^2)+\sum_{j+k+|\ell|\geq 5} \frac{V^2_{Z^j\oZ{}^k U^\ell}}{j! k! \ell!} z^j\oz^k u^\ell,
\\
v^3&=-\frac{{\rm i}}{2}\,(z^2\oz-z\oz^2)+\sum_{j+k+|\ell|\geq 5} \frac{V^3_{Z^j\oZ{}^k U^\ell}}{j! k! \ell!} z^j\oz^k u^\ell,
\\
v^4&=\frac{1}{4} z^2\oz^2+\sum_{j+k+|\ell|\geq 5} \frac{V^4_{Z^j\oZ{}^k U^\ell}}{j! k! \ell!} z^j\oz^k u^\ell,
\endaligned
\end{equation*}
enjoying the relations \eqref{relation-ord-5-simplified} and \eqref{relation-order-s-3} with $V^3_{Z^3\oZ}=0$. Regarding the conjugation relation, the coefficients $V_J$ of this normal form enjoy the cross-section $\mathcal K$ in \eqref{partial-cross-sec}, the normalizations \eqref{normalizations-B-extra} and the normalizations described in Lemma \ref{lem-branch-B-3-1}.
\end{Theorem}

\subsubsection{\bf Branch B-2-3}
\label{subsec-B-2-3}

Although vanishing differential invariants appearing in \eqref{Rec-rel-ord-5-B3} precludes to normalize the remaining two Maurer--Cartan forms in order five, solving the homogeneous system constituted by the coefficients of $\omega^Z, \omega^{\oZ}, \omega^{r}, r=1,\ldots,4$ in these relations reveals some crucial information. In particular, it shows that just a few lifted invariants remain nonzero in order six, among them we have
\begin{equation*}
\aligned
V^1_{Z^2\oZ^2U_1^2}=-\frac{\i}{3}\, V^1_{Z^3\oZ^2U_2}=2\,V^2_{Z^2\oZ^2U_1U_2}=4\, V^3_{Z^2\oZ U_1^3}=2\,V^4_{Z^2\oZ^2U_2^2}=-8\i V^4_{Z^2\oZ U_1^2U_2}=-8\i V^4_{Z^2\oZ U_2^3}.
\endaligned
\end{equation*}

The recurrence relation of the first lifted invariant is
\begin{equation}\label{rec-rel-B-3-3}
\aligned
&\dt V^1_{Z^2\oZ^2 U_1^2}\equiv -\frac{3}{2}\, V^1_{Z^2\oZ^2 U_1^2}\,\alpha^4_{U_4},
\endaligned
\end{equation}
which shows it as a relative invariant. Thus, we shall divide the next computations into the two further subbranches:

\begin{description}
  \item[Branch B-2-3-1] when $V^1_{Z^2\oZ^2 U_1^2} \neq 0$,
  \item[Branch B-2-3-2] when $V^1_{Z^2\oZ^2 U_1^2}$ vanishes locally.
\end{description}

Clearly in Branch {\bf B-2-3-1}, we can normalize the Maurer-Cartan form $\alpha^4_{U_4}$ by setting $\| V^1_{Z^2\oZ^2 U_1^2}\|= 1$ and solving its corresponding recurrence relation \eqref{rec-rel-B-3-3}. But, $\alpha^2_{U_3}$ remains unnormalized yet. In order to check possible normalization of this Maurer-Cartan form, let us consider the following {\it most symmetric} (cf. \cite{Olver-2009, Heyd-2024}) surface of this branch
\begin{equation*}
\aligned
v_1&=z\oz\pm\frac{1}{8}\,z^2\oz^2u_1^2+\frac{\i}{4}\,(z^3\oz^2u_2-z^2\oz^3u_2),
\\
v_2&=\frac{1}{2}\,(z^2\oz+z\oz^2)+\frac{1}{8}\,z^2\oz^2u_1u_2,
\\
v_3&=-\frac{\i}{2} (z^2\oz-z\oz^2)+\frac{1}{48}\,(z^2\oz u_1^3+z\oz^2u_1^3),
\\
v_4&= \frac{1}{4}\, z^2\oz^2+\frac{1}{16}\,z^2\oz^2u_2^2+\frac{\i}{32}\,(z^2\oz u_1^2u_2-z\oz^2 u_1^2u_2),
\endaligned
\end{equation*}
which arises after vanishing identically all unconsidered lifted invariants. Among the CR manifolds in this branch, this surface has the maximal dimension. However, our computations reveals that this isotropy group is in fact trivial. Consequently, the remaining Maurer-Cartan form $\alpha^2_{U_3}$ can be normalized in a certain order $\geq 7$. In view of the overwhelming complexity of the required computations, we do not aim to chase the precise place of occurring this normalization and conclude this branch by the obtained {\it partial} moving frame.

\begin{Theorem}
Every $6$-dimensional totally nondegenerate CR submanifold $M\subset\mathbb C^4$ belonging to Branch {\bf B-2-3-1} can be mapped, through some origin-preserving transformation, to the {\sl partial} normal form
\begin{align*}
v_1&=z\oz\pm\frac{1}{8}\,z^2\oz^2u_1^2+\frac{\i}{4}\,(z^3\oz^2u_2-z^2\oz^3u_2)+\frac{V^1_{Z^2\oZ^2U_2^2}}{8}\,z^2\oz^2u_2^2+\sum_{j+k+|\ell|\geq 7}\, \frac{V^1_{Z^j\oZ^kU^\ell}}{j! k! \ell!}z^j\oz^ku^\ell,
\\
v_2&=\frac{1}{2}\,(z^2\oz+z\oz^2)+\frac{1}{8}\,z^2\oz^2u_1u_2+\sum_{j+k+|\ell|\geq 7}\, \frac{V^2_{Z^j\oZ^kU^\ell}}{j! k! \ell!}z^j\oz^ku^\ell,
\\
v_3&=-\frac{\i}{2} (z^2\oz-z\oz^2)+\frac{1}{48}\,(z^2\oz u_1^3+z\oz^2u_1^3)+\sum_{j+k+|\ell|\geq 7}\, \frac{V^3_{Z^j\oZ^kU^\ell}}{j! k! \ell!}z^j\oz^ku^\ell,
\\
v_4&= \frac{1}{4}\, z^2\oz^2+\frac{1}{16}\,z^2\oz^2u_2^2+\frac{\i}{32}\,(z^2\oz u_1^2u_2-z\oz^2 u_1^2u_2)+\frac{\i V^4_{Z^2\oZ^2U_2^2}}{16}\,(z^2\oz u_2^3-z\oz^2u_2^3)
\\
& \ \ \ \ \ \ \ \ \ \ \ \ \ \ \ \ \ \ \ \ \ \ \ \ \ \ \ \ \ \ \ \ \ \ \ \ \ \ \ \ \ \ \ \ \ \ \ \ \ +\sum_{j+k+|\ell|\geq 7}\, \frac{V^2_{Z^j\oZ^kU^\ell}}{j! k! \ell!}z^j\oz^ku^\ell,
\end{align*}
where, regarding the conjugation relation, the coefficients $V_J$ of this normal form enjoy the cross-section $\mathcal K$ in \eqref{partial-cross-sec} supplemented by the normalizations \eqref{normalizations-B-extra}.
\end{Theorem}

In the second subbranch {\bf B-2-3-2}, the recurrence relation \eqref{rec-rel-B-3-3} is of course of no use to normalize any of the remaining Maurer-Cartan forms $\alpha^4_{U_4}$ and $\alpha^2_{U_3}$. It turns out that normalizing these forms is indeed impossible even in higher dimensions. More precisely, by applying the solution of the homogeneous system, mentioned at the beginning of this subsection \ref{subsec-B-2-3}, one finds the structure equations \eqref{struc-equ} as
\begin{equation}\label{struc-eq-B-3-3}
\aligned
\dt \omega^Z &= \bigg(\frac{1}{4}\,\alpha^4_{U_4}-{\rm i}\,\alpha^2_{U_3}\bigg)\wedge\omega^Z, \qquad \qquad
\dt \omega^{\oZ} = \bigg(\frac{1}{4}\,\alpha^4_{U_4}+{\rm i}\,\alpha^2_{U_3}\bigg)\wedge\omega^{\oZ},
\\
\dt \omega^1 &= 2\i \omega^Z\wedge\omega^{\oZ}+\frac{1}{2}\,\alpha^4_{U_4}\wedge\omega^1,
\\
\dt \omega^2 &= -\big(\omega^Z+\omega^{\oZ}\big)\wedge\omega^1+\frac{3}{4}\,\alpha^4_{U_4}\wedge\omega^2+\alpha^2_{U_3}\wedge\omega^3,
\\
\dt \omega^3 &= \!\i \big(\omega^Z-\omega^{\oZ}\big)\wedge\omega^1 - \alpha^2_{U_3}\wedge\omega^2+\frac{3}{4}\,\alpha^4_{U_4}\wedge\omega^3,
\\
\dt \omega^4 &= -\frac{1}{2}\big(\omega^Z+\omega^{\oZ}\big)\wedge\omega^2+\frac{\rm i}{2}\big(\omega^Z-\omega^{\oZ}\big)\wedge\omega^3 +\alpha^4_{U_4}\wedge\omega^4.
\endaligned
\end{equation}
These structure equations are of constant type and thus, by the principles of classical Cartan's theory, no further normalizations of the (pseudo-)group parameters is available. Moreover, the equivalence problem of the manifolds $M$ belonging to this subbranch is identified by the equivalence problem to the $8$-dimensional {\it prolonged space} $M\times\mathcal G^{\sf red}$, where $\mathcal G^{\sf red}$ is a $2$-dimensional reduced subgroup of $\mathcal G$ determined by the remained Maurer--Cartan forms $\alpha^2_{U_3}$ and $\alpha^4_{U_4}$. At this moment, the structure equations of these Maurer--Cartan forms are trivially
\begin{equation}\label{alpha-B-3-3}
\dt \alpha^4_{U_4}=0, \qquad  \dt\alpha^2_{U_3}=0.
\end{equation}

Thus, the equivalence problem to the CR manifolds $M$ belonging to this branch is encoded by the structure equations \eqref{struc-eq-B-3-3}-\eqref{alpha-B-3-3}. As these equations are of constant type, then each two arbitrary such manifolds are biholomorphically equivalent. Recall that in this branch, we have in particular Beloshpka's totally nondegenerate model $M(0,\frac{1}{4})$ with the defining equations (cf. \eqref{def-eq-model})
\begin{equation}
\label{Bel-model-B-3-3}
\aligned
v^1&=z\oz,
\\
v^2&=\frac{1}{2}\,(z^2\oz+z\oz^2),
\\
v^3&=-\frac{\i}{2} (z^2\oz-z\oz^2),
\\
v^4&= \frac{1}{4}\, z^2\oz^2,
\endaligned
\end{equation}
with the $2$-dimensional isotropy group constructed by the infinitesimal {\it dilation} and {\it rotation} vector fields (see \eqref{dilation} and \eqref{rotation} for their corresponding flows)
\begin{equation}
\label{X1-X2}
\aligned
X_1&=z\,\partial_z+2\,w_1\,\partial_{w_1}+3\,w_2\,\partial_{w_2}+3\,w_3\,\partial_{w_3}+4\,w_4\,\partial_{w_4},
\\
X_2&=-\i z\,\partial_z-w_3\,\partial_{w_2}+w_2\,\partial_{w_3},
\endaligned
\end{equation}
which correspond to the two remained unnormalized Maurer--Cartan forms, \cite{Valiquette-SIGMA}. Summing up, we therefore have

\begin{Theorem}
Every $6$-dimensional totally nondegenerate submanifold of $\mathbb C^4$ belonging to the subbranch {\bf B-2-3-2} can be mapped through some origin-preserving holomorphic transformation to Beloshapka's model surface $M(0,\frac{1}{4})$. The isotropy group of these manifolds is $2$-dimensional, generated infinitesimally by the vector fields \eqref{X1-X2}. The mentioned normal form transformation is unique up to the action of this isotropy group.
\end{Theorem}

\section{Convergence of the normal forms}
\label{sec-convergence}

Having finalized the construction of the desired normal forms, we now turn to the proof of their convergence. For this purpose, we employ the recent criterion established in \cite{OSV-preprint}. This criterion is rooted in the theory of {\it involutive} differential equations and relies crucially on the celebrated Cartan–K\"{a}hler theorem \cite{Seiler}.
Before stating the main result of \cite{OSV-preprint} in our present setting, let us first gather some necessary preliminaries.

A glance on the fourteen branching normal forms obtained in this paper reveals that, with the exceptions of {\bf A$'$‑2‑2‑2}, {\bf A$''$‑3‑2‑2}, and {\bf B‑2‑3‑2}, every six‑dimensional totally nondegenerate CR manifold in $\mathbb C^5$ possesses a trivial isotropy group. Biholomorphically, these three exceptional branches include only the Beloshapka's model surfaces \eqref{model-A-2-2-2}, \eqref{model-A'-2-2-2} and \eqref{Bel-model-B-3-3}, which are obviously real‑analytic. Consequently, these three branches may be excluded from the subsequent convergence analysis, and we may henceforth assume that the isotropy groups of all CR manifolds in question, are trivial, as is required by the convergence theorem of \cite{OSV-preprint}.

In this section, we verify the convergence of the normal forms arising in Branch {\bf A$'$}. The same argument also proves convergence for the normal forms in the other branches.

As in \cite[Example 10.5]{OSV-preprint}, the determining Cauchy-Riemann PDE system \eqref{eq: determining equations} for the holomorphic pseudo-group $\mathcal G$ fails to be involutive (we refer the reader to \cite{Seiler, OSV-preprint} for definition). But, applying the standard real-to-complex linear change of variables
\begin{equation*}
u_j=\frac{w^j+\ow^j}{2} \qquad {\rm and} \qquad v^j=\frac{w^j-\ow^j}{2\i}, \qquad j=1,\ldots 4,
\end{equation*}
converts it to the involutive complex PDE system
\begin{equation}
\label{eq-CR-complex}
Z_{\oz}=Z_{\ow^j}=W^k_{\oz}=W^k_{\ow^j}=0, \qquad j,k=1, \ldots, 4.
\end{equation}
In order to check the involution of this system, we assign the order
\begin{equation}\label{order}
w^4\prec w^3\prec w^2\prec w^1\prec z\prec \oz\prec\ow^4\prec\ow^3\prec\ow^2\prec\ow^1
\end{equation}
to its independent variables. For every
 $j=1, \ldots, 10$, let ${\sf x}_j$ be the $j$-th variable appearing in the above increasing ordering chain \eqref{order}.
  We then define {\it Cartan index} ${\sf b}^{(j)}$ to be the total number of derivative operators $\partial_{{\sf x}_j}$ that occur in the determining system \eqref{eq-CR-complex}:
\[
{\sf b}^{(1)}=\cdots={\sf b}^{(5)}=0, \qquad {\sf b}^{(6)}=\cdots={\sf b}^{(10)}=5.
\]
On the other hand, prolonging the system \eqref{eq-CR-complex} to the second order gives rise to $200$ algebraically independent differential equations. Hence, the rank of the second order prolonged system is $r_2=200$. Consequently, the system is involutive, as it satisfies the {\it Cartan involutivity test}
\[
\sum_{j=1}^{10} \, j \cdot {\sf b}^{(j)} = 200 = r_2
\]
and there is no integrability condition defined on it.

 The cross-section $\mathcal K_{\bf A^\prime}$ associated to Branch {\bf A}$^\prime$ is
\begin{equation}
\label{cross-section-A'}
\aligned
\mathcal K_{\bf A^\prime}=\bigg\{& v^1_{z\oz}=v^2_{z^2\oz}=1, \qquad v^3_{z^2\oz}=-\i, \qquad v^4_{z^3\oz}=\i\bigg\} \ \ \bigcup
\\
\bigg\{&v^\kappa_{z^{j}u^\ell}=v^\kappa_{z\oz u^\ell}=v^2_{z^2\oz u^\ell}=v^3_{z^2\oz u_2^j u_3^k u_4^l}=v^4_{z^2\oz u_3^k u_4^l}=\re v^4_{z^2\oz u_2^{j+1} u_3^k u_4^l}=v^1_{z^{3+j}\oz u^\ell}
\\
= & v^2_{z^2\oz^2 u_2^j u_3^k u_4^l}=v^2_{z^3\oz u^\ell}=v^3_{z^2\oz^2 u_3^k u_4^l}=v^1_{z^2\oz^2 u_3^k u_4^l}=v^3_{z^2\oz u_1 u_2^j u_3^k u_4^l}=v^4_{z^3\oz u_3^k u_4^l}=v^3_{z^3\oz u_4^l}
\\
= & v^3_{z^3\oz u_1 u_4^l}= \im v^2_{z^3\oz^2 u_4^l}=v^1_{z^2\oz^2 u_2 u_3^k u_4^l}=\re v^3_{z^3\oz u_2 u_4^l}=\im v^4_{z^2\oz u_2 u_4^l}=0\bigg\},
\endaligned
\end{equation}
for $j,k,l\geq 0, \ell\in\mathbb N^4_0$ and $\kappa=1,\ldots, 4$. This cross-section is added in the subsequent subbranches of {\bf A}$^\prime$ by an additional equality arising from the normalization of $\alpha^4_{U_4}$. However, this {\it finite number} of extra equations does not affect the argument. Also notice that the conjugate jets have not been included in $\mathcal K_{\bf A^\prime}$ as they will be recovered automatically through the normalization process of the {\it real} defining functions of the surfaces, without requiring any further attempt.

Since the Cauchy-Riemann PDE system \eqref{eq-CR-complex} becomes involutive in the complex coordinates $z, \oz, w^j, \ow^j$, $j=1,\ldots, 4$, we have to recast it in terms of these new variables. For this purpose, we rewrite the real defining equations
\[
v^\kappa=v^\kappa(z,\oz,u_1, \ldots, u_4) \qquad \kappa=1,\ldots, 4,
\]
of the CR manifold $M$ in terms of the complex coordinates,
\begin{equation}
\label{def-eq-complex}
\ow^\kappa(z,\oz,w)=w^\kappa-2\i v^\kappa\big(z,\oz, \frac{w^1+\ow^1(z,\oz,w)}{2}, \ldots, \frac{w^4+\ow^4(z,\oz,w)}{2}\big) \qquad \kappa=1,\ldots,4,
\end{equation}
which can be solved for $\ow^\kappa$s by means of the implicit function theorem. In this new coordinates, then we regard $\ow^1, \ldots, \ow^4$ as the dependent variables while $z, \oz, w^1, \ldots, w^4$ play the role of independent ones. Thus, whereas in the former coordinates the jets were of the forme $v^\kappa_{z^j \oz^k u^\ell}$, they now take the form $\ow^\kappa_{z^j \oz^k w^\ell}$.

Implicit differentiation of the equations \eqref{def-eq-complex} produces expressions for the new jet coordinates $\ow^\kappa_J$ in terms of the original ones $v^\kappa_J$. At the first order level, for instance, one obtains --- employing the Einstein summation convention --- the expressions
\begin{equation}
\label{first-order}
\aligned
\ow^\kappa_z=-2\i v^\kappa_z-\i v^\kappa_{u_j}\ow^j_z, \qquad \ow^\kappa_{w^r}=\delta^\kappa_r-\i v^\kappa_{u_r}-\i v^\kappa_{u_j}\ow^j_{w^r}, \qquad \kappa, r=1,\ldots,4,
\endaligned
\end{equation}
where $\delta^\kappa_r$ is the Kronecker delta.
Now, evaluating the above equation at the first order cross-section elements $v^\kappa_z=v^\kappa_{\oz}=v^\kappa_{u_r}=0$ in \eqref{cross-section-A'} gives rise to
\[
\ow^\kappa_z=0, \qquad {\rm and} \qquad \ow^\kappa_{w^r}=\delta^\kappa_r,
\]
In addition, an induction on the successive derivations of the equations \eqref{first-order} with respect to the variables $z, w^1, \ldots, w^4$ and upon applying the cross-section conditions \eqref{cross-section-A'} yields that
\[
\ow^\kappa_{z^jw^\ell}=0,
\]
for every $j\geq 0, \neq\ell\in\mathbb N^4_0$ with $(j, \ell)\neq (0,0)$.

Next, in order two, our computations show for $\kappa=1,\ldots,4$ that
\begin{equation}\label{second-order}
\aligned
\ow^\kappa_{z\oz}=-2\i \big(v^\kappa_{z\oz}+\frac{1}{2}\,v^\kappa_{zu_j}\ow^j_{\oz}\big)-\i \big(v^\kappa_{u_j}\ow^j_{z\oz}+v^\kappa_{\oz u_j}\ow^j_{z}+\frac{1}{2}\,v^\kappa_{u_ju_l}\ow^j_z\ow^l_{\oz}\big).
\endaligned
\end{equation}
Examining these equations upon the cross-section \eqref{cross-section-A'} immediately implies that
\[
\ow^1_{z\oz}=-2\i, \qquad \ow^2_{z\oz}=\ow^3_{z\oz}=\ow^4_{z\oz}=0.
\]
Again, applying a certain induction on the equations \eqref{second-order} gives in addition that
\[
\ow^\kappa_{z\oz w^\ell}=0
\]
for every nonzero vector $\ell\in\mathbb N^4_0$. Proceeding analogously for computing $\ow^\kappa_{z^2\oz}, \ow^\kappa_{z^2\oz^2}, \ow^\kappa_{z^3\oz}, \ldots$ and applying the appropriate inductions, we ultimately obtain the following translation of the cross-section \eqref{cross-section-A'} in terms of the new complex coordinates --- for the notational convenience, henceforth we write $w_1, \ldots, w_4$ instead of $w^1, \ldots, w^4$ in the indices
\begin{equation}
\label{cross-section-A'-complex-coord}
\aligned
\widetilde{\mathcal K}_{\bf A^\prime}=\bigg\{&\ow^\kappa_{w_\kappa}=1, \qquad \ow^1_{z\oz}=\ow^2_{z^2\oz}=-2\i, \qquad \ow^3_{z^2\oz}=-2, \qquad \ow^4_{z^4\oz}=2 \bigg\} \ \ \bigcup
\\
\bigg\{& \ow^\kappa_{z^j w^\ell}=\ow^\kappa_{z\oz w^\ell}=\ow^2_{z^2\oz w^\ell}=\ow^3_{z^2\oz w_2^j w_3^k w_4^l}=\ow^4_{z^2\oz w_3^k w_4^l}=\ow^4_{z^2\oz w_2^{j+1} w_3^k w_4^l}+\ow^4_{z\oz^2 w_2^{j+1} w_3^k w_4^l}
\\
=&w^1_{z^{3+j}\oz w^\ell}=\ow^2_{z^2\oz^2 w_2^j w_3^k w_4^l}=\ow^2_{z^3\oz w^\ell}=\ow^3_{z^2\oz^2 w_3^k w_4^l}=\ow^1_{z^2\oz^2 w_3^k w_4^l}=\ow^3_{z^2\oz w_1 w_2^j w_3^k w_4^l}
\\
=&\ow^4_{z^3\oz w_3^k w_4^l}=\ow^3_{z^3\oz w_4^l}=\ow^3_{z^3\oz w_1 w_4^l}= \ow^2_{z^3\oz^2 w_4^l}-\ow^2_{z^2\oz^3 w_4^l}=\ow^1_{z^2\oz^2 w_2 w_3^k w_4^l}
\\
& \ \ \ \ \ \ \ \ \ \ \ \ \ \ \ \ \ \ \ \ \ \ \ \ \ \ \ \ \ \ \ \ \ \ \ \ \ \ \ \ \ \ \ \ \ \ =\ow^3_{z^3\oz w_2 w_4^l}+\ow^3_{z\oz^3 w_2 w_4^l}=\ow^4_{z^2\oz w_2 w_4^l}-\ow^4_{z\oz^2 w_2 w_4^l}=0\bigg\},
\endaligned
\end{equation}
for $j,k,l\geq 0, \ell\in\mathbb N^4_0$ and $\kappa=1,\ldots, 4$.

With the formal basis symbols ${\bf e}_1, \ldots, {\bf e}_4$, define the {\it index set}
\[
{\bf I}_{{\bf A^\prime}}:=\mathcal C^1 \, {\bf e}_1\biguplus\mathcal C^2 \, {\bf e}_2\biguplus\mathcal C^3 \, {\bf e}_3\biguplus\mathcal C^4 \, {\bf e}_4,
\]
of the cross-section $\widetilde{\mathcal K}_{\bf A^\prime}$, where for each $\nu=1, \ldots, 4$, $\mathcal C^\nu$ denotes the set of indices of $\ow^\nu$ which are visible in $\widetilde{\mathcal K}_{\bf A^\prime}$. For example, two equations
\[
\ow^3_{z^2\oz w_2^j w_3^k w_4^l}=0 \qquad {\rm and} \qquad \ow^4_{z^2\oz w_2 w_4^l}-\ow^4_{z\oz^2 w_2 w_4^l}=0
\]
in $\widetilde{\mathcal K}_{\bf A^\prime}$ correspond to
\[
z^2\oz w_2^j w_3^k w_4^l\in\mathcal C^3 \qquad {\rm and} \qquad z^2\oz w_2 w_4^l-z\oz^2 w_2 w_4^l\in\mathcal C^4.
\]
The degree of an element $f\,{\bf e}_\nu\in {\bf I}_{{\bf A^\prime}}$ is defined as the degree of $f$ as a monomial or binomial. For each $n\geq 0$, we also denote ${\bf I}^{(n)}_{{\bf A^\prime}}$ the set of all {\it symbol} monomials and binomials in ${\bf I}_{{\bf A^\prime}}$ which are of (homogeneous) degree $\geq n$.

\begin{Definition}
In accordance with the given order $w^4\prec w^3\prec w^2\prec w^1\prec z\prec \oz$ in \eqref{order}, we say a monomial
\[
w_4^{i_1}w_3^{i_2}w_2^{i_3}w_1^{i_4}z^{i_5}\oz^{i_6}\,{\bf e}_\nu\in {\bf I}_{{\bf A^\prime}}
\]
is of {\it class} $\kappa$ if it is the smallest index such that $i_\kappa\neq 0$. The class of a binomial in ${\bf I}_{{\bf A^\prime}}$ is defined as the {\sl equal} class of its two monomials. For a given monomial (binomial) $m\in {\bf I}_{{\bf A^\prime}}$ of class $\kappa$, the {\it (Pommaret) involutive cone} $\mathcal C(m)$ is the set of all monomials (binomials) constituted by multiplying $m$ with monomials of class $\leq\kappa$.

\end{Definition}

As is shown in \cite{OSV-preprint}, the corresponding normal form to a {\it minimal} cross-section converges whenever for some $n\geq 0$, the associated index set ${\bf I}^{(n)}$ possesses a {\it Rees decomposition}; meaning that ${\bf I}^{(n)}$ can be expressed as a finite disjoint union of involutive cones. More precisely, we actually showed in \cite{OSV-preprint} that the existence of such decomposition for ${\bf I}^{(n)}$ guarantees that the initial conditions, identified by the cross-section for the associated normal form differential system, satisfy the hypothesis of the Cartan-K\"{a}hler theorem \cite[Theorem 9.4.1]{Seiler}.

\begin{Remark}
Although the main results of \cite{OSV-preprint} are formulated in terms of coordinate cross-sections, it is well-known that the linear span of the index set $\bf I$ admits always a {\it monomial} basis (see the paragraph preceding to \cite[Proposition 5.1.3]{Seiler}). These monomial generators in turn correspond to a coordinate cross-section which is equivalent to the original one.
\end{Remark}

In our case, the index set ${\bf I}^{(6)}_{\bf A'}$ admits the Rees decomposition
\begin{equation}
\label{Rees-dec}
\footnotesize
\aligned
{\bf I}^{(6)}_{\bf A'}=&\biguplus_{\alpha=1}^4\bigg(\biguplus_{j+|\ell|=6}\,\mathcal C(z^jw^\ell)\biguplus_{|\ell| = 4}\,\mathcal C(z\oz w^\ell)\bigg]
\bigg)\,{\bf e}_{\alpha}
\\
&\biguplus\bigg(\biguplus_{j+|\ell|=2}\,\mathcal C(z^{3+j}\oz w^\ell)\biguplus_{j+k = 2}\,\mathcal C(z^2\oz^2 w_3^{j} w_{4}^{k})\biguplus_{k = 3}^4\,\mathcal C(z^2\oz^2 w_2 w_k)\bigg)\,{\bf e}_1
\\
& \biguplus \bigg(\biguplus_{|\ell| = 3}\,\mathcal C(z^2\oz w^\ell)\biguplus_{j+k+l= 2}\,\mathcal C(z^2\oz^2 w_2^jw_3^{k} w_{4}^{l})\biguplus_{|\ell| = 2}\,\mathcal C(z^3\oz w^\ell)\biguplus\,\mathcal C(z^3\oz^2 w_4-z^2\oz^3 w_4)\bigg)\,{\bf e}_2
\\
& \biguplus \bigg(\biguplus_{i=0,1 \atop j+k+l = 3}\,\mathcal C(z^2\oz w_1 w_2^jw_3^kw_4^l)\biguplus_{k+l=2}\,\mathcal C(z^2\oz^2 w_3^kw_4^l)\biguplus_{j=1}^2\,\mathcal C(z^3\oz w_1^{2-j} w_4^j)
\biguplus\,\mathcal C(z^3\oz w_2 w_4+z\oz^3w_2w_4)\bigg)\,{\bf e}_3
\\
&\biguplus \bigg(\biguplus_{k+l=3}\,\mathcal C(z^2\oz w_3^k w_4^l)\biguplus_{k+l=2}\,\mathcal C(z^3\oz w_3^k w_4^l)\biguplus_{j+k+l=2}\,\mathcal C(z^2\oz w_2^{j+1}w_3^k w_4^l+z\oz^2 w_2^{j+1}w_3^k w_4^l)
\\
&\ \ \ \ \ \ \ \ \ \ \ \biguplus\,\mathcal C(z^2\oz w_2 w_4^2-z\oz^2 w_2 w_4^2)\bigg)\,{\bf e}_4.
\endaligned
\end{equation}
By virtue of the fact that our constructed cross-section $\widetilde{\mathcal K}_{\bf A^\prime}$ is minimal, the existence of the Rees decomposition \eqref{Rees-dec} ensures the convergence of its associated normal form, as was claimed.

\subsection*{Acknowledgments}
The author, gratefully acknowledges the crucial guidance and invaluable remarks provided by Peter Olver and Francis Valiquette throughout the preparation of this paper. Their worth discussions played a pivotal role in shaping this paper. He also thanks sincerely Jo\"{e}l Merker for his helpful comments and discussions. The research of the author was supported, in part, by the Iran
National Science Foundation (INSF), under the project No. 4031893, and the Institute for Research in Fundamental Science (IPM), grant No.\ 1401510415.

\end{document}